%% file: ms.tex
\pdfoutput=1
\RequirePackage{amsmath}
\newif\ifSpringer
\Springerfalse 

\ifSpringer
\RequirePackage{fix-cm}
\documentclass[smallextended]{svjour3}
\smartqed	
\usepackage{etoolbox}
\AtEndEnvironment{proof}{\phantom{}\qed}
\else
\documentclass[11pt]{scrartcl}
\usepackage{amsthm}
\newtheorem{theorem}{Theorem}[section]
\newtheorem{lemma}{Lemma}[section]
\newtheorem{remark}{Remark}[section]

\fi

\usepackage[utf8]{inputenc}
\usepackage[T1]{fontenc}
\usepackage{lmodern}	
\usepackage[english]{babel}

\usepackage[draft=false,kerning=true]{microtype}
\usepackage[numbers]{natbib} 

\usepackage{amsmath,amssymb,graphicx}

\usepackage{adjustbox}
\usepackage{rotating}
\usepackage{float}
\usepackage[caption = false]{subfig}
\usepackage[ruled,vlined,linesnumbered]{algorithm2e}

\usepackage{booktabs,tabularx} 
\usepackage{siunitx} 
\usepackage{pgfplots}
\usepackage{pgfplotstable} 
\usepackage{longtable}

\usepackage{subfig}
\usepackage{tablefootnote}

\usepackage[title]{appendix}

\usepackage[colorinlistoftodos,prependcaption,textsize=tiny]{todonotes}

\usepackage{scalerel}
\usepackage{tikz}
\usetikzlibrary{svg.path}

\usepackage[normalem]{ulem} 

\DeclareOldFontCommand{\bf}{\normalfont\bfseries}{\mathbf}
\DeclareMathOperator*{\argmax}{arg\,max}
\DeclareMathOperator*{\argmin}{arg\,min}

\newcommand{\R}{\mathbb{R}}
\newcommand{\sn}{{\mathcal S}^n}
\newcommand{\inprod}[2]{{\langle #1,#2 \rangle}} 
\newcommand{\trace}{\textrm{trace}}

\newcommand{\nosemic}{\renewcommand{\@endalgocfline}{\relax}}
\newcommand{\dosemic}{\renewcommand{\@endalgocfline}{\algocf@endline}}
\let\oldnl\nl
\newcommand{\nonl}{\renewcommand{\nl}{\let\nl\oldnl}}

\newcommand{\centerhfill}[1][\quad]{\hspace{\stretch{0.5}}#1\hspace{\stretch{0.5}}}

\newcommand{\rew}{\textcolor{black}}

\definecolor{orcidlogocol}{HTML}{A6CE39}
\tikzset{
	orcidlogo/.pic={
		\fill[orcidlogocol] svg{M256,128c0,70.7-57.3,128-128,128C57.3,256,0,198.7,0,128C0,57.3,57.3,0,128,0C198.7,0,256,57.3,256,128z};
		\fill[white] svg{M86.3,186.2H70.9V79.1h15.4v48.4V186.2z}
		svg{M108.9,79.1h41.6c39.6,0,57,28.3,57,53.6c0,27.5-21.5,53.6-56.8,53.6h-41.8V79.1z M124.3,172.4h24.5c34.9,0,42.9-26.5,42.9-39.7c0-21.5-13.7-39.7-43.7-39.7h-23.7V172.4z}
		svg{M88.7,56.8c0,5.5-4.5,10.1-10.1,10.1c-5.6,0-10.1-4.6-10.1-10.1c0-5.6,4.5-10.1,10.1-10.1C84.2,46.7,88.7,51.3,88.7,56.8z};
	}
}
\providecommand{\keywords}[1]
{
	\small	
	\textbf{\textit{Keywords---}} #1
}

\newcommand\orcidicon[1]{\href{https://orcid.org/#1}{\mbox{\scalerel*{
				\begin{tikzpicture}[yscale=-1,transform shape]
				\pic{orcidlogo};
				\end{tikzpicture}
			}{|}}}}
\usepackage{hyperref}

\pgfplotsset{compat=1.15}
\pgfplotstableset{
	begin table=\begin{longtable},
		end table=\end{longtable},
}


\begin{document}

\title{SDP-based bounds for graph partition via extended ADMM\thanks{This project has received funding from the European Union’s Horizon 2020 research and innovation programme under the Marie Sk\l{}odowska-Curie grant agreement MINOA No 764759.}
}

\ifSpringer
\else
\author{Angelika Wiegele \orcidicon{0000-0003-1670-7951} 
	 \footnote{Institut f\"ur Mathematik, Alpen-Adria-Universit\"at Klagenfurt, Universit\"atsstraße 65-67, 9020, \href{mailto:angelika.wiegele@aau.at}{angelika.wiegele@aau.at}, \href{mailto:shudian.zhao@aau.at}{shudian.zhao@aau.at}} \,
			 \and Shudian Zhao \orcidicon{0000-0001-6352-0968} $^\dagger$\footnote{Corresponding author}
}

\fi


\maketitle

\begin{abstract}

We study two NP-complete graph partition problems, $k$-equipartition problems and graph partition problems with knapsack constraints (GPKC). We introduce tight SDP relaxations with nonnegativity constraints to get lower bounds, the SDP relaxations are solved by an extended alternating direction method of multipliers (ADMM). In this way, we obtain high quality lower bounds for $k$-equipartition on large instances up to $n =1000$ vertices within as few as five minutes and for GPKC problems up to $n=500$ vertices within as little as one hour. On the other hand, interior point methods fail to solve instances from $n=300$ due to memory requirements. 
We also design heuristics to generate upper bounds from the SDP solutions, giving us tighter upper bounds than other methods proposed in the literature with low computational expense.

\ifSpringer
\else
\keywords{Graph partitioning, Semidefinite programming, ADMM, Combinatorial optimization}
\fi

\end{abstract}

\section{Introduction}

Graph partition problems have gained importance recently due to their applications in the area of engineering and computer science such as telecommunication~\cite{lisser2003graph} and parallel computing~\cite{hendrickson2000graph}. The solution of a graph partition problem would serve to partition the vertices of a graph $G(V,E)$ into several groups under certain constraints for capacity or cardinality in each group. The optimal solution is expected to have the smallest total weight  of cut edges. This problem is NP-complete~\cite{garey1974some}. Previous studies  worked on improving quadratic programming or linear programming formulations to reduce the computational expense with commercial solvers~\cite{fan2010linear}.  
Relaxations are also used to approximate this problem. \citet{garey1974some} were the first to use eigenvalue and eigenvector information to get relaxations for graph partition. \citet{ghaddar2011branch} have recently used a branch-and-cut algorithm based on SDP relaxations to compute global optimal solutions of $k$-partition problems.

The $k$-equipartition problem, which is to find a partition with the minimal total weight of cut edges and the vertex set $V$ is equally partitioned into $k$ groups, is one of the most popular graph partition problems. Another problem that interests us is the graph partition problem under knapsack constraints (GPKC). In the GPKC, each vertex of the graph network is assigned a weight and the knapsack constraint \rew{needs} to be satisfied in each group.

\citet{lisser2003graph} compared various semidefinite programming (SDP) relaxations and linear programming (LP) relaxations for $k$-equipartition problems. They showed that the nonnegativity constraints become dominating in the SDP relaxation when $k$ increases. However, the application of this formulation is limited by the computing power of SDP solvers, because adding all the sign constraints for a symmetric matrix of dimension $n$ causes $O(n^2)$ new constraints and entails a huge computational burden especially for large instances.  \citet{nguyen2016contributions} proposed a tight LP relaxation for GPKC problems and a heuristic to build upper bounds as well. Semidefinite programming has shown advantages in generating tight lower bounds for quadratic problems with knapsack constraints \cite{helmberg1996quadratic} and $k$-equipartition problem, but so far there have been no attempts to apply SDP relaxations to GPKC.

Algorithms for solving SDPs have been intensively studied in the previous years. \citet{malick2009regularization} designed the boundary point method to solve SDP problems with equations. It can solve instances with a huge number of constraints that interior point methods (IPMs) fail to solve. This method falls into the class of alternating direction method of multipliers (ADMM). ADMM has been studied in the area of convex optimization and been proved of linear convergence when one of the objective terms is strongly convex \cite{nishihara2015general}. In recent years, there have been studies focusing on generalizing this idea on solving convex optimizations with more blocks of variables. \citet{chen2016direct}, for example, have proved the convergence of 3-block ADMM on certain scenarios, but the question as to whether the direct extension 3-block ADMM on SDP problems is convergent is still open.

There have also been varied ideas about combining other approaches with ADMM for solving SDP problems. \citet{de2018using} added the dual factorization in the ADMM update scheme, while \citet{sun2019sdpnal+} combined ADMM with Newton's methods. Both attempts have improved the performance of the algorithms.

\subsection*{Main results and outline}
In this paper, we will introduce an extended ADMM algorithm and apply it to the tight SDP relaxations for graph partition problems with nonnegativity constraints. We will also introduce heuristics to obtain a feasible partition from the solution of the SDP relaxation.

This paper is structured as follows. In Section~\ref{sec:gpp}, we will introduce two graph partition problems, the $k$-equipartition problem and the graph partition problems with knapsack constraints (GPKC). We will discuss different SDP relaxations for both problems. In Section~\ref{sec:admm}, we will design an extended ADMM and illustrate its advantages in solving large SDP problems with nonnegativity constraints. In Section~\ref{sec:post-processing},  we will introduce two post-processing methods used to generate lower bounds using the output from the extended ADMM. In Section~\ref{sec:up_bound}, we will design heuristics to build a tight upper bound from the SDP solution to address the original problem. Numerical results of experiments carried out on graphs with different sizes and densities will be presented in Section~\ref{sec:tests}.
Section~\ref{sec:conclusion} concludes the paper.

\subsection*{Notation}
We define by $e_n$ the vector of all ones of length $n$,
by $\mathbf{0}_n$ the vector of all zeros of length $n$ and by $\mathbf{0}_{n\times n}$ the square matrix of all zeros of dimension $n$. We omit the subscript in case the dimension is clear from the context.
The notation $[n]$ stands for the set of integers $\{1,\dots,n\}$.
Let $\sn$ denote the set of all $n\times n$ real symmetric matrices. We denote by $M\succeq 0$ that the matrix $M$ is positive semidefinite and let ${\mathcal S}_+^n$ be the set of all positive semidefinite matrices of order
$n\times n$. We denote by $\inprod{\cdot}{\cdot}$ the 
trace inner product. That is, for any $M, N \in \R^{n\times n}$, we define $\inprod{M}{N}:= \trace (M^\top N )$. 
Its associated norm is the Frobenius norm, denoted by $\| M\|_F := \sqrt{\trace (M^\top M )}$.
We denote by $\textrm{diag}(M)$ the operation of getting the diagonal entries of matrix $M$ as a vector.
The projection on the cone of positive semidefinite matrices is denoted by $\mathcal{P}_{\succeq 0}(\cdot)$. The projection onto the interval $[L,U]$ is denoted by $\mathcal{P}_{[L,U]}$.
We denote by $\lambda(\cdot)$ the eigenvalues. That is, for any $M \in \R^{n\times n}$, we define $\lambda(M)$ the set of all eigenvalues of $M$. Also, we denote $\lambda_{\max}(\cdot)$ the largest eigenvalue. 
We denote by $x \sim U(0,1)$ a variable $x$ from uniform distribution between 0 and 1.
We define by $\textrm{argmaxk}(\cdot,s)$ the index set of the $s$ largest elements.

\section{Graph partition problems}\label{sec:gpp}
\subsection{$k$-equipartition problem}
For a graph $G(V,E)$, the $k$-equipartition problem is the problem of finding an equipartition of the vertices in $V$ with $k$ groups that has the minimal total weight of edges cut by this partition. The problem can be described with binary variables,
\begin{equation}\label{eq:keq-miqp}
\begin{aligned}
\min&~ \frac{1}{2}\langle L,Y Y^\top\rangle \\
\textrm{s.t.}~& Y e_k = e_n, \\
& Y^\top  e_n = m e_k,\\
&Y_{ij} \in \{ 0, 1 \}, \forall i \in [n], j\in [k], 
\end{aligned}
\end{equation}
where $L$ is the Laplacian  matrix for $G$, variable $Y \in \R^{n\times k}$ indicates which group each vertex is assigned to and $e_n$ (resp. $e_k$) is the all-one vector of dimension $n$ (resp. $k$).

This problem is NP-hard and \citet{lisser2003graph} proposed the SDP relaxation
\begin{equation}\label{eq:keq-sdp}
\begin{aligned}
\min~ & \frac{1}{2}\langle L,X \rangle\\
    \textrm{s.t.}~& \textrm{diag}(X) = e,\\
    & X  e = m e, \\
    & X \succeq 0,
\end{aligned}
\end{equation}
where $X \in \mathcal{S}_n$, and $e \in \R^n$ is the all-one vector.

To tighten this SDP relaxation, we can add more inequalities to problem~\eqref{eq:keq-sdp}. Here, we introduce two common inequalities for SDP relaxations derived from $0/1$ problems. 

The process of relaxing $Y Y^\top$ to $X$ implies that $X$ is a nonnegative matrix, hence the first group of inequalities we consider is $X \geq 0$ and the corresponding new SDP relaxation is

\begin{equation}\label{eq:keq-dnn}
\begin{aligned}
\min~ & \frac{1}{2}\langle L,X \rangle\\
    \textrm{s.t.}~& \textrm{diag}(X) = e,\\
    & X e = me,\\
    & X \succeq 0,\\
    & X \geq 0.
\end{aligned}
\end{equation}
\rew{This kind of SDP is also called a doubly nonnegative program (DNN) since the matrix variable is both, positive semidefinite and elementwise nonnegative.}

Another observation is the following. For any vertices triple $(i,j,k)$, if vertices~$i$ and $j$ are in the same group, and vertices~$j$ and $k$ are in the same group, then vertices $i$ and $k$ must be in the same group. This can be modeled by the transitivity constraints~\cite{lisser2003graph} given as follows
\begin{equation*}
  \textrm{MET}:= \{X=(X_{ij}) \mid X_{ij}+ X_{ik} \leq 1 + X_{jk},\forall i,j,k \in [n]\}.
\end{equation*}
The set formed by these inequalities is the so-called metric polytop. Adding the transitivity constraints to the SDP relaxation~(\ref{eq:keq-dnn}) gives
\begin{equation}\label{eq:keq-met}
\begin{aligned}
\min~& \frac{1}{2}\langle L, X \rangle \\
\textrm{s.t.}~& \textrm{diag}(X) = e,\\
& X e = m e,\\
& X \succeq 0,\\
& X \geq 0,\\
& X \in \textrm{MET}.
\end{aligned}
\end{equation}

\subsection{Graph partition problem under knapsack constraints (GPKC)}

Given a graph $G(V,E)$ with nonnegative weights on the vertices and a capacity bound~$W$, the GPKC asks to partition the vertices such that the total weight of cut edges is minimized and the total weight of vertices in each group does not exceed the capacity bound~$W$.

A mathematical programming formulation is given as 
\begin{equation}\label{eq:gpkc-miqp}
    \begin{aligned}
        \min ~& \frac{1}{2} \langle L, YY^\top  \rangle\\
         \textrm{s.t.}~& Y e_n = e_n,\\
        & Y^\top a \leq We_n,\\
        & Y_{ij} \in \{0,1\}^{n\times n},\forall i \in [n], j\in [n],
    \end{aligned}
\end{equation}
where $Y\in \R^{n\times n}$, $a\in\R^n$ is the vertex weight vector and $W$ is the capacity bound.
We assume $a_i \leq W~\forall i \in [n]$, otherwise the problem is infeasible.
Again, we can derive the SDP relaxation
\begin{equation}\label{eq:gpkc-sdp}
\begin{aligned}
    \min~& \frac{1}{2}\langle L,X\rangle\\
    \textrm{s.t.}~&  \textrm{diag}(X) = e,\\
    &  X a\leq W e, \\
    & X \succeq 0.
\end{aligned}
\end{equation}
Similar as the $k$-equipartition problem, we can tighten the relaxation by imposing sign constraints, i.e.,
\begin{equation}\label{eq:gpkc-dnn}
\begin{aligned}
    \min~&\frac{1}{2} \langle L,X\rangle\\
    \textrm{s.t.}~&  \textrm{diag}(X) = e,\\
    & X a\leq W e,\\
    & X \succeq 0,\\
    & X \geq 0,
\end{aligned}
\end{equation}
and additionally by imposing the transitivity constraints which gives
\begin{equation}\label{eq:gpkc-met}
\begin{aligned}
\min~& \frac{1}{2}\langle L,X\rangle\\
\textrm{s.t.}~&  \textrm{diag}(X) = e,\\
& X a\leq W e,\\
& X \succeq 0,\\
& X \geq 0,\\
& X \in \textrm{MET}.
\end{aligned}
\end{equation}

\section{Extended ADMM}\label{sec:admm}
The SDP relaxations introduced in Section~\ref{sec:gpp} have a huge number of constraints, even for medium-sized graphs.
The total number of sign constraints for $X$ is $O(n^2)$ and adding the constraint $X \in \textrm{MET}$ in the SDP relaxations causes $3\binom{n}{3}$ extra constraints.
Therefore, solving these tight relaxations is out of reach for state-of-the-art algorithms like interior point methods (IPMs).
However, finding high quality lower bounds by tight SDP relaxations for graph partition problems motivates us to develop an efficient algorithm that can deal with SDP problems with inequalities and sign constraints on large-scale instances.
Since the 2-block alternating direction method of multiplier (ADMM) has shown efficiency in solving large-scale instances that interior point methods fail to solve, we are encouraged to extend this algorithm for SDP problems with inequalities in the form
\begin{equation}\label{eq:sdp-general-p}
        \begin{aligned}
        \min~&\langle C, X \rangle\\
             \text{s.t.}~&\mathcal{A}(X) = b, \\
                            &\mathcal{B}(X) = s, \\
                & X \succeq 0,\\
                & L \leq X \leq U,\\
                & l\leq s \leq u,
        \end{aligned}
\end{equation}
where $C \in \sn$, $\mathcal{A}: \sn \rightarrow \R^m$, $\mathcal{B}: \sn \rightarrow \R^q$, $b \in \R^m$, $l, u \in \R^q$. We have the slack variable $s \in \R^q$ to form the inequality constraints, and $l$ and $u$ can be set to $-\infty$ and $+\infty$ respectively. Also, $L \in \sn $ and $U \in \sn$ can be symmetric matrices filled with all elements as $-\infty$ and $+\infty$ respectively. That makes formulation~\eqref{eq:sdp-general-p} able to represent SDP problems with any equality and inequality constraints. This formulation is inspired by the work of \citet{sun2019sdpnal+}.
\rew{All semidefinite programs given above fit into this formulation. E.g., in~\eqref{eq:keq-dnn} operator $\mathcal{A}$ includes the diagonal-constraint and the constraint $Xe=me$, the operator $\mathcal{B}$ as well as the variables $s$ are not present, $L$ is the matrix of all zeros and $U$ the matrix having $+\infty$ everywhere.} 

Following the ideas for the 2-block ADMM \citep{malick2009regularization}, we form the update scheme in Algorithm~\ref{alg:eADMM} to solve the dual of problem~\eqref{eq:sdp-general-p}.
\begin{lemma}
	The	dual problem for \eqref{eq:sdp-general-p} is given as 
	\begin{equation}\label{eq:sdp-general-d}
	\begin{aligned}
	\max~&  b^\top y + \mathcal{F}_{1}(S) + \mathcal{F}_{2}(v)\\
	\textrm{s.t.}~&\mathcal{A}^{*}y + \mathcal{B}^{*}\bar{y} + S + Z = C,\\
	&\bar{y} = v,\\
	&Z \succeq 0,
	\end{aligned}
	\end{equation}
	where $\mathcal{F}_1(S) = \inf_{W} \{ \langle S,W \rangle \mid  L \leq W \leq U \}$ and $\mathcal{F}_2(v) = \inf_{\omega} \{ \langle v,\omega \rangle \mid l \leq \omega \leq u\}$. 
\end{lemma}
\begin{proof}
	We derive this dual problem by rewriting the primal SDP problem \eqref{eq:sdp-general-p} in a more explicit way, namely
	\begin{equation}\label{eq:sdp-general-p2}
	\begin{aligned} 
	\min~&\langle C, X \rangle\\
	\textrm{s.t.}~&   \mathcal{A}(X) = b,\\
	&   \mathcal{B}(X) -s = \mathbf{0}_q,\\
	&   X \succeq 0,\\
	&   X \geq L,\\
	&  -X \geq -U,\\
	&   s \geq l,\\
	&   -s \geq -u.\\
	\end{aligned}
	\end{equation}
	Then, the dual of \eqref{eq:sdp-general-p2} is 
	\begin{equation}\label{eq:sdp-general-d2}
	\begin{aligned}
	\max \qquad & b^\top y + \mathbf{0}_q^\top \bar{y} + \langle \mathbf{0}_{n\times n},Z \rangle  + \langle L,S_L\rangle -\langle U, S_U\rangle + l^\top v_l - u^\top v_u \\
	\textrm{s.t. } \qquad & \mathcal{A}^*y + \mathcal{B}^* \bar{y} + Z + S_L - S_U = C,\\
	& -\bar{y} + v_l - v_u = \mathbf{0}_q,\\
	& Z \succeq 0, \\
	& S_L, S_U , v_l, v_u \geq 0.
	\end{aligned}
	\end{equation}
	The following equivalences hold for each entry of the dual variables $S_L$ and $S_U$ in~\eqref{eq:sdp-general-d2}
	\begin{equation*}
	\begin{aligned}
	& X_{ij} = L_{ij} \iff S_{L,ij} \neq 0, S_{U,ij} = 0 ;\\
	& X_{ij} = U_{ij} \iff S_{U,ij} \neq 0, S_{L,ji} = 0 ;\\
	& L_{ij} < X_{ij} < U_{ij} \iff S_{L,ij} = S_{U,ij} = 0.
	\end{aligned}
	\end{equation*}
	If we let $S := S_L - S_U $, then for each entry of $S$ 
	\begin{equation*}\begin{aligned} \label{eq:S}
	& \forall i \in [n], j \in [n], S_{ij} = \begin{cases}
	S_{L,ij}, & \text{if $S_{L,ij} \neq 0$},\\
	-S_{U,ij}, & \text{otherwise}.
	\end{cases} 
	\end{aligned}
	\end{equation*}
	Thus, in the dual objective function, we have
	\begin{equation}
	\begin{aligned}
	\langle L, S_L \rangle -\langle U, S_U \rangle = \sum_{ S_{L,ij} \neq 0 } L_{ij} S_{L,ij} - \sum_{ S_{L,ij} =0 } U_{ij} S_{U,ij}.
	\end{aligned}
	\end{equation}

	Expressing $\inf_{W}\{ \langle S, W\rangle \mid L\leq W \leq U \}$ element-wisely gives
	\begin{equation}\begin{aligned}
		& \inf_{W_{ij}} \{ S_{ij} W_{ij} \mid L_{ij} \leq W_{ij} \leq U_{ij} \} = \begin{cases}
		L_{ij}S_{ij}, & \text{if}~ S_{ij} \geq 0,\\
		U_{ij}S_{ij}, & \text{if}~ S_{ij} < 0. 
              \end{cases}
              \end{aligned}
	\end{equation}
	Combining the observations above, we end up with
	\begin{equation}
	\langle L, S_L \rangle - \langle U, S_U \rangle = \inf_{W}\{ \langle S, W\rangle \mid L\leq W \leq U \}.
	\end{equation}
	Similarly, let $v: = v_l- v_u$, then we have
	\begin{equation}
	\begin{aligned}
	l^\top v_l - u^\top v_u = \sum_{v_{l,k}\neq 0} l_{k} v_{l,k} - \sum_{v_{l,k} =0} u_{k}v_{u,k} = \inf_{\omega} \{ \langle v, \omega \rangle \mid l \leq \omega \leq u \}
	\end{aligned}
	\end{equation}
	
	Hence, problem~\eqref{eq:sdp-general-d} is equivalent to~\eqref{eq:sdp-general-d2} and it is the dual of~\eqref{eq:sdp-general-p}.
\end{proof}

We now form the augmented Lagrangian function corresponding to~\eqref{eq:sdp-general-d}.
\begin{equation}
\begin{aligned}
\mathcal{L}(y, \bar{y}, Z, S,v; X, s) =~&  b^\top y +\mathcal{F}_{1}(S) + \mathcal{F}_{2}(v) \\
& - \langle \mathcal{A}^*y + \mathcal{B}^* \bar{y} + S + Z - C ,X \rangle - \langle -\bar{y} + v, s\rangle \\
& - \frac{\sigma}{2} \|\mathcal{A}^*y + \mathcal{B}^* \bar{y} + S + Z - C\|^2_F - \frac{\sigma}{2} \|-\bar{y} + v \|^2.
\end{aligned}
\end{equation}
The saddle point of this augmented Lagrangian function is
\begin{equation} \label{eq:augL}
	(y^*, \bar{y}^*, Z^*, S^*,v^*, X^*, s^*) :=  \argmin_{X,s}\argmax_{y,\bar{y},v,Z,S} \mathcal{L}(y, \bar{y}, Z, S,v; X, s),
\end{equation}
which is also an optimal solution for the primal and dual problems. If both, the primal and the dual problem, have strictly feasible points, then a point $(X, s, y, \bar{y}, Z, S, v)$ is optimal if and only if
\begin{subequations}
\begin{align}
  \mathcal{A}(X) = b,\quad \mathcal{B}(X) = s,\quad L \leq X \leq U,\quad l\leq s \leq u,\label{eq:optcond-primalfeas}\\
\mathcal{A}^{*}y + \mathcal{B}^{*}\bar{y} + S + Z = C,\quad \bar{y} = v,\label{eq:optcond-dualfeas} \\
  X \succeq 0,\quad Z \succeq 0,\quad \inprod{X}{Z} = 0, \label{eq:optcond-compl}\\
  (X_{ij} -L_{ij})(U_{ij} -X_{ij})S_{ij} = 0, ~ \forall i \in [n]~ \forall j \in [n],\quad L\leq X \leq U,\label{eq:optcond-compl-S}\\
  (s_k-l_k)(u_k-s_k)v_k = 0, ~ \forall k \in [q],\quad l \leq v \leq u \label{eq:optcond-compl-v}. 
\end{align}
\end{subequations}
\begin{remark}
	\eqref{eq:optcond-compl-S} and~\eqref{eq:optcond-compl-v} is derived from the optimality conditions for~\eqref{eq:sdp-general-p2}, namely from 
	\begin{equation}
		\begin{aligned}
			(X_{ij}-L_{ij})S_{L,ij} = 0, S_{L,ij} \geq 0, X_{ij} \geq L_{ij}, \forall i \in [n] ~\forall j \in [n],\\
			(U_{ij}-X_{ij})S_{U,ij} = 0, S_{U,ij} \geq 0,X_{ij} \leq U_{ij}, \forall i \in [n] ~\forall j \in [n],\\
			(s_{k}-l_{k})v_{l,k} = 0, v_{l,k} \geq 0, s_k \geq l_k, \forall k \in [q] ,\\
			(l_{k}-s_{k})v_{l,k} = 0, v_{l,k} \geq 0, s_k \leq u_k, \forall k \in [q] .\\
		\end{aligned}
	\end{equation}
With $S = S_L - S_U$ and $v = v_l - v_u$, we obtain~\eqref{eq:optcond-compl-S} and~\eqref{eq:optcond-compl-v}.
\end{remark}

We solve problem~\eqref{eq:augL} coordinatewise, i.e., we optimize only over a block of variables at a time while keeping all other variables fixed. The procedure is outlined in Algorithm~\ref{alg:eADMM}.

\begin{algorithm} \label{alg:eADMM}
	\SetAlgoLined
 	\nonl Initialization: Select $\sigma^k > 0$, $\varepsilon_{tol} >0 $; 	$k = 0,~X_0 = 0,~Z_0 = 0,~S_0 = 0$\;
	\nonl \While{~$\max \{\varepsilon_{pc},\varepsilon_{dc},\varepsilon_{opt_d},\varepsilon_{opt_p},\varepsilon_{pb} \} > \varepsilon_{tol}$}	
	{  
		$	(y^{k+1},\bar{y}^{k+1}) = \argmin_{(y,\bar{y})} -b^{\top}y-s^\top\bar{y}  + \langle \mathcal{A}^*y +\mathcal{B}^*\bar{y}, X^k\rangle $
		$+ \frac{\sigma^k}{2} \|\mathcal{A}^*y+\mathcal{B}^*\bar{y}+S^{k}+Z^{k}-C\|^2 + \frac{\sigma^k}{2}\|v^k-\bar{y}\|^2$\label{step:solve_y};

	$S^{k+1}  =  \argmin_{S} - \mathcal{F}_1(S)  + \langle X^k ,S\rangle  $
	$+  \frac{\sigma^k}{2}\| \mathcal{A}^*y^{k+1} +\mathcal{B}^*\bar{y}^{k+1} +S +Z^k-C \|^2_{F}$ \label{step:solve_S}\;

	$	Z^{k+1} = \argmin_{Z}  \langle X^k ,Z\rangle +  \frac{\sigma^k}{2}\| \mathcal{A}^*y^{k+1} +\mathcal{B}^*\bar{y}^{k+1} +S^{k+1} +Z-C \|^2_{F} $,
	$v^{k+1} =  \argmin_{v} - \mathcal{F}_2(v) + v^\top s^k + \frac{\sigma^k}{2}\|v-\bar{y}^{k+1}\|^2 $\label{step:solve_Z} \;
	 
	$X^{k+1} = X^k +\sigma^{k}(\mathcal{A}^*y^{k+1}+\mathcal{B}^*\bar{y}^{k+1}+S^{k+1}+Z^{k+1}-C)$, 
	$s^{k+1}= s^k + \sigma^{k}(v^{k+1}-\bar{y}^{k+1})$\label{step:solve_primal}  \;
	Update infeasibilities $ \varepsilon_{dc}$, $\varepsilon_{pc}$, 	$\varepsilon_{pb}$ ,$\varepsilon_{opt_v}$, $\varepsilon_ {opt_m}$\;
	Tune stepsize and obtain $\sigma^{k+1}$ \;
	$k\leftarrow k + 1$\;
}
\nonl	where $\varepsilon_{dc}: = \frac {\|\mathcal{A}^*y + \mathcal{B}^*\bar{y} + Z + S - C \|_F}{1+\|C\|_F} +\frac{\|-\bar{y} + v \|}{1+\|y\|}$ , $\varepsilon_{pc} :=\frac{ \|\mathcal{A}(X) -b \|}{1+\|b\|} +\frac{\|\mathcal{B}(X) - s \|}{1+\|s\|}$, $\varepsilon_{pb}:= \frac{\|X-\mathcal{P}_{[L,U]}(X)\|_F}{1+\|X\|_F}$, $\varepsilon_ {opt_m} := \frac{\|X -\mathcal{P}_{[L,U]}(X - S)\|_F}{1 + \| X\|_F + \|S\|_F}$, $\varepsilon_{opt_v}: = \frac{\|v- \mathcal{P}_{[l,u]}(v-s)\|}{1+\|v\| +\|s\|}$.
	\caption{Extended ADMM for problem \rew{\eqref{eq:sdp-general-d}}}
\end{algorithm}
In Step~\ref{step:solve_y}, the minimization over $(y, \bar{y})$, we force the first order optimality conditions to hold, i.e., we set the gradient with respect to $(y, \bar{y})$ to zero and thereby obtain the explicit expression
\begin{equation}\label{eq:lin-y}
  \begin{aligned}
  	\begin{pmatrix}
  	y^{k+1}\\
  	\bar{y}^{k+1}
  	\end{pmatrix} =\begin{pmatrix}
  	\mathcal{A}\mathcal{A}^* & \mathcal{A}\mathcal{B}^* \\
  	\mathcal{B}\mathcal{A}^* & \mathcal{B}\mathcal{B}^* + I \end{pmatrix}^{-1}
  	 \begin{pmatrix} 
  \frac{b}{\sigma^{k}}  - \mathcal{A}(S^k+Z^k-C + \frac{1}{\sigma^{k}}X^k) \\ -\mathcal{B}(S^k+Z^k-C+\frac{1}{\sigma^{k}}X^k) + v^k +\frac{1}{\sigma^{k}} s^k   \end{pmatrix}.
\end{aligned}
\end{equation}

\rew{Note that the size of $y^k$ is the number of equality constraints and the size of $\bar{y}^k$ the number of inequality constraints.
  By abuse of notation we write $\mathcal{A}\mathcal{A}^*$ for the matrix product formed by the system matrix underlying the operator $\mathcal{A}(\cdot)$.
Similarly for $\mathcal{B}\mathcal{A}^*$, $\mathcal{B}\mathcal{B}^*$.}

\rew{In practice, we solve~\eqref{eq:lin-y} in the following way. First, we apply the Cholesky decomposition 
\begin{equation}
	RR^\top = \begin{pmatrix}
	\mathcal{A}\mathcal{A}^* & \mathcal{A}\mathcal{B}^* \\
	\mathcal{B}\mathcal{A}^* & \mathcal{B}\mathcal{B}^* + I \end{pmatrix} = : Q.
\end{equation}
Since $\mathcal{A}$ and $\mathcal{B}$ are row independent, the Cholesky decomposition exists. Moreover, the Cholesky decomposition only needs to be computed once since matrix $Q$ remains the same in all iterations.
Then, we update $(y, \tilde{y})$ as 
\begin{equation}
	 \begin{aligned} RR^\top
	\begin{pmatrix}
	y^{k+1}\\
	\bar{y}^{k+1}
      \end{pmatrix} =
      rhs
      \end{aligned}
      \end{equation}
by solving two systems of equations subsequently, i.e., $R \mathbf{x} = rhs$ and then solve the system $R^\top \mathbf{y} = \mathbf{x}$ and thereby having solved $RR^\top \mathbf{y} = rhs$.
}

In Step~\ref{step:solve_S}, the minimization amounts to a projection onto the non-negative orthant.
\begin{equation}
\begin{aligned}		
S^{k+1} =&~ \argmin_{S}  -\mathcal{F}_1(S) + \langle X^k ,S\rangle +  
\frac{\sigma^k}{2}\| \mathcal{A}^*y^{k+1} +\mathcal{B}^*\bar{y}^{k+1} +S +Z^k-C \|^2_{F}, \\
S^{k+1}_{ij} =&~\begin{cases}
 \argmin_{S_{ij}\geq 0}  \| M^{k+1}_{ij}+S_{ij} -\frac{1}{\sigma^k} L_{ij} \|^2, S_{ij} \geq 0,\\
 - \argmin_{S_{ij}\leq 0}  \|M^{k+1}_{ij} -S_{ij}  + \frac{1}{\sigma^k} U_{ij} \|^2, S_{ij} \leq 0,
 \end{cases}\\
=&~\begin{cases}
\mathcal{P}_{\geq 0}(-M^{k+1}_{ij}+\frac{1}{\sigma^k} L_{ij}), S_{ij} \geq 0,\\
- \mathcal{P}_{\leq 0}(M^{k+1}_{ij} + \frac{1}{\sigma^k} U_{ij} ), S_{ij} \leq 0,
&~\end{cases} \\
=&~ \begin{cases}
 \frac{1}{\sigma^k}\mathcal{P}_{\geq L_{ij}} (\sigma^k M^{k+1}_{ij}) -M^{k+1}_{ij},S_{ij} \geq 0,\\
 \frac{1}{\sigma^k}\mathcal{P}_{\leq U_{ij}} (\sigma^k M^{k+1}_{ij}) -M^{k+1}_{ij}, S_{ij} \leq  0,
\end{cases}\\
= &\frac{1}{\sigma^k}\mathcal{P}_{[L_{ij},U_{ij}]} (\sigma^k M^{k+1}_{ij}) -M^{k+1}_{ij},
\end{aligned}
\end{equation}
where $M^{k+1}:= \mathcal{A}^*y^{k+1} + \mathcal{B}^*\bar{y}^{k+1} + Z^k + \frac{1}{\sigma^k} X^k -C$. Hence,
\begin{equation}
S^{k+1}= \frac{1}{\sigma^k}\mathcal{P}_{[L,U]} (\sigma^{k} M^{k+1}) -M^{k+1}.
\end{equation}

Similarly, in Step~\ref{step:solve_Z} for $v^{k+1}$ we have
\begin{equation}
	 v^{k+1} =  \frac{1}{\sigma^k} \mathcal{P}_{[l,u]}(\sigma^k\bar{y}^{k+1} -  s^k)  -(\bar{y}^{k+1} - \frac{1}{\sigma^k} s^{k})
\end{equation}

In Step~\ref{step:solve_Z}, for $Z^{k+1}\succeq 0$, the minimizer is found via a projection onto the cone of positive semidefinite matrices
\begin{equation}
	\begin{aligned}
		Z^{k+1} =&~ \argmin_{Z\succeq 0 } \langle X^k ,Z\rangle +  \frac{\sigma^k}{2}\| \mathcal{A}^*y^{k+1} +\mathcal{B}^*\bar{y}^{k+1} +S^{k+1} +Z-C \|^2_{F}\\
		=&~	\argmin_{Z\succeq 0 } \| \mathcal{A}^*y^{k+1} +\mathcal{B}^*\bar{y}^{k+1} +S^{k+1} +Z + \frac{1}{\sigma^k} X^k-C \|^2_{F}\\
		=&~ - \mathcal{P}_{\preceq 0} (N^{k+1}),
	\end{aligned}
\end{equation}
where $N^{k+1}:= \mathcal{A}^*y^{k+1} +\mathcal{B}^*\bar{y}^{k+1} +S^{k+1} + \frac{1}{\sigma^k}X^k-C $.

Finally, by substituting $Z^{k+1}$ and $v^{k+1}$ into Step~\ref{step:solve_primal} we obtain  
\begin{equation}
\begin{aligned}
X^{k+1} &=\sigma^k \mathcal{P}_{\succeq 0} (N^{k+1}),\\
s^{k+1} &=   \mathcal{P}_{[l,u]}(\sigma^k\bar{y}^{k+1} -  s^k).
\end{aligned}
\end{equation}
\begin{remark}
  Throughout Algorithm~\ref{alg:eADMM}, the complementary slackness condition for $(X^{k+1},Z^{k+1})$ holds.
  This is since $\frac{1}{\sigma^k}X^{k+1}$ is the projection onto the positive semidefinite cone of matrix $N^{k+1}$, while $-Z^{k+1}$ is a projection onto the negative semidefinite cone of the same  matrix $N^{k+1}$.
\end{remark}

\subsection{Stepsize adjustment}

Previous numerical results showed that the practical performance of an ADMM is strongly influenced by the stepsize $\sigma$. The most common way is to adjust $\sigma$ to balance primal and dual infeasibilities: if $\varepsilon_p/ \varepsilon_d < c$ for some constant $c$, then increase $\sigma$; if $\varepsilon_p/ \varepsilon_d > \frac{1}{c}$, then decrease $\sigma$.

\citet{lorenz2018non} derived an adaptive stepsize for the Douglas-Rachford Splitting (DRS) scheme. In the setting of the 2-block ADMM, this translates to the ratio between norms of the primal and dual variables $\frac{\|X_k\|}{\|Z_k\|}$ in the $k$-th iteration. In general, for the 2-block ADMM this update rule yields a better performance than the former one.

In this paper, we use either of these update rules, depending on the type of problems we solve. For SDP problems with equations and nonnegativity constraints only, we apply the adaptive stepsize method from~\citet{lorenz2018non} since it works very well in practice.
However, the situation is different for SDP problems with inequalities different than nonnegativity constraints. In this case, we use the classic method to adjust the stepsize $\sigma$ according to the ratio between the primal and dual infeasibilities.

\section{Lower bound post-processing algorithms} \label{sec:post-processing}

We relax the original graph partition problem to an SDP problem, thereby generating a lower bound on the original problem.
However, when solving the SDP by a first-order method, it is hard to reach a solution to high precision in reasonable computational time. Therefore, we stop the ADMM already when a medium precision is reached. In this way, however, the solution obtained by the ADMM is not always a safe underestimate for the optimal solution of the SDP problem. Hence, we need a post-processing algorithm that produces a safe underestimate for the SDP relaxation, which is then also a lower bound for the graph partition problem.

Theorem~\ref{the:lb_sdp_gen} leads to the first post-processing algorithm. 
Before stating it, we rewrite Lemma~3.1 from~\cite{jansson2007rigorous} in our context.

\begin{lemma} \label{lemma1}
Let $X, Z \in \mathcal{S}_n$, and  $  0 \leq \lambda(X)\leq \bar{x} $, where $\lambda(\cdot)$ indicates the operation of getting the eigenvalues of the respective matrix. Then
\begin{equation*}
    \langle X, Z \rangle \geq \sum_{\lambda(Z) < 0} \bar{x} \cdot \lambda(Z).
\end{equation*}
\end{lemma}

%


\begin{theorem} \label{the:lb_sdp_gen}
	Given $Z\in \sn $, $y \in \R^m$, $\tilde{y} \in \R^q $, $\tilde{v} \in \mathbb{R}^q$, $\tilde{S} \in \sn$, and let $X  \in \sn_+ $ be an optimal solution for \eqref{eq:sdp-general-p} and $\bar{x} \geq \lambda_{\max}(X)$, then we have a safe lower bound for the optimal value $p^*$
	\begin{equation}\label{eq:rigorous_lb_general}
	 \textrm{lb}:= b^\top \tilde{y} + \mathcal{F}_1(\tilde{S}) + \mathcal{F}_2(\tilde{v})  + \sum_{\lambda(Z) <0} \bar{x} \lambda(Z).
	\end{equation}
\end{theorem}
\begin{proof}
We recall the alternative formulation of~\eqref{eq:sdp-general-p}, 
\begin{equation}\tag{\ref{eq:sdp-general-p2}}
\begin{aligned} 
    \min~&\langle C, X \rangle\\
          \textrm{s.t.}~&   \mathcal{A}(X) = b,\\
                        &   \mathcal{B}(X) -s = 0,\\
                        &   X \succeq 0,\\
                        &   X \geq L,\\
                        &  -X \geq -U,\\
                        &   s \geq l,\\
                        &   -s \geq -u,\\
\end{aligned}
\end{equation}
and the corresponding dual problem
    \begin{equation}\tag{\ref{eq:sdp-general-d2}}
        \begin{aligned}
        \max \qquad & b^\top y +\mathbf{0}_q^\top \bar{y} +\langle \mathbf{0}_{n\times n}, Z \rangle + \langle L, S_L \rangle -\langle U,S_U  \rangle+ l^\top v_l - u^\top v_u \\
         \textrm{s.t. } \qquad      & \mathcal{A}^*y + \mathcal{B}^* \bar{y} + Z + S_L - S_U = C,\\
                & -\bar{y} + v_l - v_u = 0,\\
                & Z \succeq 0, \\
                & S_L, S_U , v_l, v_u \geq 0.
        \end{aligned}
    \end{equation}
Given an optimal solution $X^*$ from \eqref{eq:sdp-general-p2} and the free variable $\tilde{y}$ and nonnegative variables $(\tilde{v}_l,\tilde{v}_u,\tilde{S}_L, \tilde{S}_U )$,  we define $Z: = C - \mathcal{A}^*\tilde{y}- \mathcal{B}^*\tilde{v}_l + \mathcal{B}^* \tilde{v}_u  - S_L + S_U$ and have
\begin{subequations}
\begin{align}
    & \langle C,X^*\rangle - (b^\top \tilde{y} +  l^\top \tilde{v}_l - u^\top \tilde{v}_u+ \langle L, \tilde{S}_L \rangle -\langle  U, \tilde{S}_U \rangle),\\
    =&~ \langle C,X^*\rangle - \langle \mathcal{A}^*\tilde{y},X^*\rangle - (l^\top \tilde{v}_l - u^\top \tilde{v}_u+  \langle L, \tilde{S}_L\rangle - \langle U,\tilde{S}_U\rangle),\\
    \geq&~ \langle C,X^*\rangle - \langle \mathcal{A}^*\tilde{y},X^*\rangle -  \langle \mathcal{B}^*\tilde{v}_l,X^* \rangle +  \langle \mathcal{B}^*\tilde{v}_u,X^* \rangle - \langle X, \tilde{S}_L \rangle + \langle U, \tilde{S}_U\rangle, \label{ineq:lb-general}\\
    =&~ \langle C - \mathcal{A}^*\tilde{y}- \mathcal{B}^*\tilde{v}_l + \mathcal{B}^* \tilde{v}_u  - S_L + S_U, X^*\rangle,\\
    =&~ \langle Z, X^*\rangle,\\
    \geq & \sum_{\lambda(Z) <0} \bar{x} \lambda(Z),
\end{align}
\end{subequations}
where inequality~\eqref{ineq:lb-general} holds because $\tilde{v}_l,\tilde{v}_u,\tilde{S}_U$ and $\tilde{S}_L$ are nonnegative. 
This gives us a lower bound for the problem~\eqref{eq:sdp-general-p2} as 
\begin{equation} \label{lb-general-sdp}
    \textrm{lb}:= b^\top \tilde{y} + l^\top \tilde{v}_l - u^\top \tilde{v}_u+ \langle L, \tilde{S}_L \rangle - U\cdot \tilde{S}_U + \sum_{\lambda(Z) <0} \bar{x} \lambda(Z).
\end{equation}
On substituting $S:= S_L-S_U$ and $v: = v_l- v_u$ into the objective function we have 
  \begin{equation}
  \begin{aligned}
  \langle L, S_L \rangle -\langle U, S_U \rangle  = \inf_{W}\{ \langle S, W\rangle \mid L\leq W \leq U \},\\
  l^\top v_l - u^\top v_u  = \inf_{\omega} \{ \langle v, \omega \rangle \mid l \leq \omega \leq u \}.
  \end{aligned}
  \end{equation}
Consequently, we can rewrite~\eqref{lb-general-sdp} as 
\begin{equation}\tag{\ref{eq:rigorous_lb_general}}
\textrm{lb}:= b^\top \tilde{y} + \mathcal{F}_1(\tilde{S}) + \mathcal{F}_2(\tilde{v})  + \sum_{\lambda(Z) <0} \bar{x} \lambda(Z).
\end{equation}
\end{proof}

For specifically structured SDP problems, a value of  $\bar{x}$ might be known.
Otherwise, without any information about an upper bound $\bar{x}$ in \eqref{eq:rigorous_lb_general} for the maximal eigenvalue $\lambda_{\max}(X)$, we approximate $\bar{x}$ as $\lambda_{\max}(\tilde{X})$ where the output from the extended ADMM is $(\tilde{X},\tilde{y}, \tilde{v},\tilde{S})$. Then, we scale it with $\mu > 1$, \rew{e.g., $\mu = 1.1$,} to have a safe bound $\mu \bar{x}$.
\rew{Note that this requires that the solution of the extended ADMM, i.e., Algorithm~\ref{alg:eADMM}, is satisfied with reasonable accuracy, say $\varepsilon = 10^{-5}$.}

The complete post-processing algorithm is summarized in Algorithm~\ref{alg:lb-eADMM}.

\begin{algorithm} \label{alg:lb-eADMM}
\SetAlgoLined
\KwIn{Data $P = (\mathcal{A},\mathcal{B}, b,l,u,L,U, C)$, approximate primal and dual optimal solution $(\tilde{X},\tilde{y},\tilde{v},\tilde{S})$ for $P$ and \rew{${\bar x}$ with $\max (\lambda(\tilde{X})) \le {\bar x}$.}}
\KwOut{Lower bound $d^*$}
Compute $d_0^* := b^{\top}\tilde{y}+\mathcal{F}_1(\tilde{S}) + \mathcal{F}_2(\tilde{v})$ \;
\eIf{$\min(\lambda(Z))< 0$} 
{ perturbation = ${\bar x} \cdot \sum_{\lambda(Z) < 0} \lambda(Z)$\;}
{perturbation = 0\;}
$d^* := d_0^* + \text{perturbation}$\; \caption{Rigorous lower bound for \eqref{eq:sdp-general-p}}
where $\mathcal{F}_1(S) = \inf_{W} \{ \langle S,W \rangle \mid  L \leq W \leq U \}$, $\mathcal{F}_2(v) = \inf_{\omega} \{ \langle v,\omega \rangle \mid l \leq \omega \leq u\}$.
\end{algorithm}
As for the $k$-equipartition problem \eqref{eq:keq-dnn}, we have $X \preceq m\cdot I$ for any feasible solution $X$. Hence, we let $\bar{x} = m$ when applying post-processing Algorithm~\ref{alg:lb-eADMM} for $k$-equipartition problems.
\rew{As for the GPKC, we have no value $\bar{x}$ at hand.}

Another way to get a safe lower bound for \eqref{eq:sdp-general-p} is to tune the output results and get a feasible solution for its dual problem \eqref{eq:sdp-general-d}.
This is outlined as Algorithm~\ref{alg:post-process3}.
The brief idea is to build a feasible solution $(y_{new},v_{new},Z_{new},S_{new})$ from an approximate solution $(\tilde{y},\tilde{v},\tilde{Z},\tilde{S})$. To guarantee feasibility of $(y_{new},v_{new}$, $Z_{new},S_{new})$, we first get a $Z_{new}$ by projecting $\tilde{Z}$ on the cone of positive semidefinite matrices. We then keep $Z_{new}$ fixed and hence have a linear problem. The final step is to find the optimal solution for this linear programming problem. 

In Algorithm~\ref{alg:eADMM}, the condition $Z \succeq 0$ is guaranteed by the projection operation onto the cone of positive semidefinite matrices. Hence, we can skip Step~\ref{eq:Zcheck} in Algorithm~\ref{alg:post-process3}.

We would like to remark that the linear program can be infeasible, but this algorithm works well when the input solution has a good precision.
The comparisons of numerical results of these two post processing algorithms are given in Section~\ref{sec:numerical-post-proc}.

\begin{algorithm}  \label{alg:post-process3}
\SetAlgoLined
\KwIn{Data $P = (\mathcal{A},\mathcal{B}, b,l,u,L,U, C)$, approximate primal and dual optimal solution $\tilde{Z}$ for $P$.}
\KwOut{Lower bound $d^*$}
Update $\tilde{Z}$: $\tilde{Z}\rightarrow \mathcal{P}_{\succeq 0} (\tilde{Z})$ \label{eq:Zcheck}\;
\eIf{
	LP problem 
\begin{equation*}
\begin{aligned}
LP(P) :=\max_{v_l,v_u,S_L,S_U \geq 0,y} \{b^\top y + l^\top v_l -u^\top v_u + \langle L, S_L \rangle - \langle U, S_U \rangle\\
 \mid \mathcal{A}^*y + \mathcal{B}^*v_l-\mathcal{B}^*v_u + S_L - S_U = C-\tilde{Z} \}
 \end{aligned}
\end{equation*} is feasible\;}
{return lower bound $d^* = LP(P)$\;}{return $d^* = - \infty$\;}
\caption{Adjusted lower bound for \eqref{eq:sdp-general-p2}}
\end{algorithm}

\section{Building upper bounds from the SDP solutions}\label{sec:up_bound}

Computing upper bounds of a minimization problem is typically done via finding feasible solutions of the original problem by heuristics.

A $k$-equipartition problem can be transformed into a quadratic assignment problem (QAP), and we can find feasible solutions for a QAP by simulated annealing (SA), see, e.g., \cite{sotirov2012sdp}. 
However, this method comes with a high computational expense for large graphs. Moreover, it cannot be generalized to GPKC problems.

Here we consider building upper bounds from the optimizer of the SDP relaxations. We apply different rounding strategies to the solution $X$ of the SDP relaxations presented in Section~\ref{sec:gpp}.

\subsection{Randomized algorithm for $k$-equipartition}

The first heuristic is a hyperplane rounding algorithm that is inspired by the Goemans and Williamson algorithm for the max-cut problem \citep{goemans1995improved} and \citet{frieze1995improved}'s improved randomized rounding algorithm  for $k$-cut problems.

Note that the Goemans and Williamson algorithm as well as the Frieze and Jerrum algorithm are designed for cut-problems formed as models on variables in $\{-1/(k-1),1\}^n$, while our graph partition problems are formed on $\{0,1\}^n$. Therefore, we need to transform the SDP solutions of problems~\eqref{eq:keq-dnn} and~\eqref{eq:gpkc-dnn} before applying the hyperplane rounding procedure.
Our hyperplane rounding algorithm for $k$-equipartition is given in Algorithm~\ref{alg:rand-keq}.
\begin{algorithm} \label{alg:rand-keq}
\SetAlgoLined
\KwData{number of partitions $k$, cluster cardinality $m$, number of sampling $M$, objective matrix $C$;}

\KwIn{Optimal solution $X \in \mathcal{S}^n$ from the SDP relaxation;}
\KwOut{$X^*  \in \{0,1\}^{n\times n}$, partition $\mathcal{P}:= \{P_t\mid t=1,\dots,k\}$, $b^*_{up}$\;}
    Initialization: $P_t \leftarrow \emptyset~ \forall t=1,\dots,k$, $b_{up}^* = + \infty$\;
    Transformation: $X \leftarrow (k X-ee^\top)/(k-1)$\;
    Get $V$ by matrix decomposition such that $V V^\top = X$\;
    \For{$\textrm{iter} =1,\dots,M$}
    {$r:=[r_{ij}] \sim U(0,1)~,\forall i\in [n]~ j \in [k]$ \;
    \For{ $t =1,\dots,k$}
    {	$P_t \leftarrow \textrm{argmaxk}_{i\in[n]/ \cup_{S\in \mathcal{P}} S} (v_i^\top r_{t},m) $\;
    \For{ $i \in P_t$}{
     $$X'_{ij}= \begin{cases}1, &j \in P_t, \\0, &j \notin P_t; \end{cases},~X'_{ji}= \begin{cases}1, &j \in P_t, \\0, &j \notin P_t. \end{cases}$$
     }}
	$b'_{up} = \langle C,X' \rangle $\;
	\If{ $b'_{up} < b^*_{up}$ }
	{ $X^* \leftarrow X'$\;
		$b^*_{up} \leftarrow b'_{up}$\;}
	}
 where $\textrm{argmaxk}_{i\in I}(a_i,s)$ returns the index set of $s$ largest elements in~$a_i$, $\forall i \in I$.
    \caption{Hyperplane rounding algorithm (Hyp) for $k$-equipartition problem}
\end{algorithm}


\subsection{Vector clustering algorithm for $k$-equipartition}
We next propose a heuristic via the idea of vector clustering. Given a feasible solution $X$ of \eqref{eq:keq-dnn}, we can get $V \in \R^{n\times n}$ with $V V^\top=X$. Let $v_i$ be the $i$-th row of $V$ and associate it with vertex $i$ in the graph. The problem of building a feasible solution from $X$ can then be interpreted as the problem of clustering vectors $v_1, \dots, v_n$ into $k$ groups. 
This can be done heuristically as follows.
\begin{enumerate}
    \item Form a new group with an unassigned vector.
    \item Select its $m-1$ closest unassigned neighbors and add them in the same group.
    \item Update the status of those vectors as assigned.
\end{enumerate}
This process is repeated $k-1$ times until all vectors are assigned in a group, yielding a $k$-equipartition for the vertices in $V$. The details are given in Algorithm~\ref{alg:rounding-keq}.
\begin{algorithm} \label{alg:rounding-keq}
	\SetAlgoLined
	\KwData{number of partitions $k$, cluster cardinality $m$, maximum iteration number $M$;}
	\KwIn{SDP relaxation  optimal solution $X \in \mathcal{S}^n$;}
	\KwOut{$X^*$, partition $\mathcal{P}:= \{P_t \mid t=1,\dots,k\}$, $b^*_{up}$\;}
	Initialization: $P_t \leftarrow \emptyset~ \forall t=1,\dots,k$, $b_{up}^* = + \infty $ .\\
	\For{ iter $ =1,\cdots,M$}
	{
	
	\For{ $t= 1,\dots, k$}
    {
   	$ i = \textrm{random} \{i \in [n] \mid i \notin \cup_{S\in \mathcal{P}} S\}$,\\
     $P_t \leftarrow \{i\};$\\    
     $P_t \leftarrow$ $P_t \cup \textrm{argmaxk}_{j \in [n]/\cup_{S\in \mathcal{P}} S}(x_i^\top x_j,m-1) $ ;\\
     \For{ $i \in P_t$}{
     $$X'_{ij}= \begin{cases}1, &j \in P_t, \\0, &j \notin P_t; \end{cases},~X'_{ji}= \begin{cases}1, &j \in P_t, \\0, &j \notin P_t. \end{cases}$$
     }   
     }
 	$b'_{up} = \langle C,X' \rangle $\;
 	\If{ $b'_{up} < b^*_{up}$ }
 	{ $X^* \leftarrow X'$\;
 	$b^*_{up} \leftarrow b'_{up}$\;}
 }
     where $\textrm{argmaxk}_{i\in I}(a_i,s)$ returns the index set of $s$ largest elements in~$a_i$, $\forall i \in I$.
	\caption{Vector clustering algorithm (Vc) for $k$-equipartition problem }
\end{algorithm}

\subsubsection{Measure closeness between vertices}

We explain in this section how we determine the closest neighbor for a vector. The idea of vector clustering is to have vectors with more similarities in the same group. In our setting, we need a measure to define the similarity between two vectors according to the SDP solution.

For a pair of unit vectors $v_i$, $ v_j$, using the relationship $\cos \measuredangle (v_i,v_j)  = v_i^\top v_j $ one can measure the angle between $v_i$ and $v_j$.

By the setting $V V^\top = X$, we have for any $i \in [n]$
\begin{equation}
	x_i =
	\begin{pmatrix}
	v_i^\top v_1 \\ \vdots \\v_i^\top v_n
	\end{pmatrix}=
	\begin{pmatrix}
	\cos \measuredangle (v_i,v_1) \\ \vdots \\ \cos \measuredangle (v_i,v_n)
	\end{pmatrix},
\end{equation}
where $x_i$ is the $i$-th row vector in $X$.

Hence, $x_i$ consists of the cosines of the angle between $v_i$ and other vectors. 
We define $\textrm{sim}(v_i,v_j): = \sum_{k=1}^{n} \cos\measuredangle (v_i,v_k)   \cos\measuredangle (v_j,v_k)  = x_i^\top x_j$ and use this as a measure in Algorithm \ref{alg:rounding-keq}.
In other words, we measures the closeness between $v_i$ and $v_j$ by \rew{their geometric} relationships with other vectors.

In Algorithm~\ref{alg:rounding-keq}, we choose a vector as the center of its group and then find vectors surrounding it and assign them to this group.

In each iteration we randomly choose one vector to be the center.

\subsection{Vector clustering algorithms for GPKC}\label{sec:ub_gpkc}
Using similar ideas as in Algorithm~\ref{alg:rounding-keq}, we construct a rounding algorithm (see Algorithm~\ref{alg:rounding-gpkc1}) for GPKC as follows.
\begin{enumerate}
    \item In each iteration, randomly choose an unassigned vector $v_i$ to start with.
    \item Add vectors in the group of $v_i$ in the order according to $\textrm{sim}(v_i,v_j)$, $\forall j\neq i \in [n]$, until the capacity constraint is violated.
    \item If no more vector fits into the group, then this group is completed and we start forming a new group.
\end{enumerate}

\begin{algorithm}  \label{alg:rounding-gpkc1}
	\SetAlgoLined
	\KwData{vertex weight $a$, knapsack bound $W$, maximum iteration number $M$\;}
	\KwIn{SDP relaxation  optimal solution $X$;}
	\KwOut{$X^* \in \{0,1\}^{n \times n}$, partition  $\mathcal{P}:= \{P_t \mid t=1,\dots,n\}$, $b^*_{up}$;}
	Initialization:$P_t \leftarrow \emptyset~ \forall t=1,\dots,n$,  $V_0\leftarrow [n]$, $t \leftarrow 1$.\\
		\For{ iter $ =1,\cdots,M$}
	{
	\While{$V_0 \neq \emptyset$}{
		$i = \textrm{random} \{i \in [n] \mid i \notin \cup_{S\in \mathcal{P}} S \}$\;
     $P_t \leftarrow \{i\}$\;
     $w_t \leftarrow a_i$\;
     $w_0 \leftarrow 0$\;
     $I \leftarrow \{i \in [n]\mid i \notin \cup_{S\in \mathcal{P}} S\}$\;
     \For{iter $=1,\dots,|I|$}{
      $j \leftarrow \argmax_{j \in I}  \langle x_j,x_i\rangle$ \;
      $w_0 \leftarrow w_t + a_j$\;
      	$I \leftarrow I / \{j\}$\;
      \eIf{$w_0 \leq W$ }{$w_t \leftarrow w_0$\; 
      	$P_t \leftarrow P_t \cup \{ j\}$\;
      }{continue.}}
      $V_0 \leftarrow V_0/\cup_{S\in \mathcal{P}} S$\;
     \For{ $i \in P_t$}{
     $$X'_{ij}= \begin{cases}1, &j \in P_t, \\0, &j \notin P_t; \end{cases},~X'_{ji}= \begin{cases}1, &j \in P_t, \\0, &j \notin P_t. \end{cases}$$ 
     }
 	$t \leftarrow t +1$ \;
    }$b'_{up} = \langle C,X' \rangle $\;
\If{ $b'_{up} < b^*_{up}$ }
{ $X^* \leftarrow X'$\;
	$b^*_{up} \leftarrow b'_{up}$\;}
}
	\caption{Vector clustering algorithm (Vc) for GPKC problem}
\end{algorithm}

\subsection{2-opt for graph partition problems}

2-opt heuristics are used to boost solution qualities for various combinatorial problems, e.g., TSP~\citep{lin1965computer}. We apply this method after running our rounding algorithms for the graph partition problems to improve the upper bounds. According to the rounding method we choose, the hybrid strategies are named as Hyperplane+2opt (also short as Hyp+2opt) and Vc+2opt for Algorithms~\ref{alg:rand-keq} and~\ref{alg:rounding-keq}, respectively.

The 2-opt heuristic for bisection problems is outlined in Algorithm~\ref{alg:2opt_keq}.
Given a partition with more than two groups, we apply 2-opt on a pair of groups $(P_s,P_t)$, which is randomly chosen from all groups in the partition, and repeat it on a different pair of groups until no more improvement can be found.

For GPKC, some adjustments are needed because of the capacity constraints. We only traverse among swaps of vertices that still give feasible solutions to find the best swap that improves the objective function value.

\begin{algorithm} \label{alg:2opt_keq}
	\SetAlgoLined
	\KwData{Lapacian matrix $L$, threshold $\varepsilon$\;}
	\KwIn{A feasible bisection $\mathcal{P}_0=\{P_1,P_2\}$ for graph $G$\;}
	\KwOut{New bisection $\mathcal{P}^*$ \;}
	$(s,t)\leftarrow \argmax_{i \in P_1,j \in P_2} \sum_{k\neq i, k \in P_1}L_{jk} - \sum_{k\neq j, k \in P_2}L_{jk} + \sum_{ k\neq j, k\in P_2 }L_{ik}  -\sum_{k\neq i, k \in P_1}L_{ik}$ \;
	$\Delta_{cost} \leftarrow \sum_{k\neq s, k \in P_1}L_{tk} - \sum_{k\neq t, k \in P_2}L_{tk} + \sum_{ k\neq t, k\in P_2 }L_{sk}  -\sum_{k\neq s, k \in P_1}L_{sk}$\;
	\While{$\Delta_{cost} > \varepsilon $}{
	$P_1 \leftarrow P_1 -\{s\} +\{t\}  $\;
	$P_2 \leftarrow P_2 -\{t\} +\{s\}  $\;
$(s,t)\leftarrow \argmax_{i \in P_1,j \in P_2} \sum_{k\neq i, k \in P_1}L_{jk} - \sum_{k\neq j, k \in P_2}L_{jk} + \sum_{ k\neq j, k\in P_2 }L_{ik}  -\sum_{k\neq i, k \in P_1}L_{ik}$\;
	$\Delta_{cost} \leftarrow \sum_{k\neq s, k \in P_1}L_{tk} - \sum_{k\neq t, k \in P_2}L_{tk} + \sum_{ k\neq t, k\in P_2 }L_{sk}  -\sum_{k\neq s, k \in P_1}L_{sk}$\;
	}
	$\mathcal{P}^* \leftarrow \{P_1,P_2\}$\;
	\caption{2-opt method for bisection problems}
\end{algorithm}

\section{Numerical results}\label{sec:tests}

We implemented all the algorithms in MATLAB and run the numerical experiments on a ThinkPad-X1-Carbon-6th with 8 Intel(R) Core(TM) i7-8550U CPU @ 1.80GHz.
The maximum iterations for extended ADMM is set to be 20\;000 and the stopping tolerance $\varepsilon_{tol}$ is set to be $10^{-5}$  by default. 

The code can be downloaded from \href{https://github.com/shudianzhao/ADMM-GP}{https://github.com/shudianzhao/ADMM-GP}.

\subsection{Instances}

In order to evaluate the performance of our algorithms, we run numerical experiments on several classes of instances.
All instances can be downloaded from  \href{https://github.com/shudianzhao/ADMM-GP}{https://github.com/ shudianzhao/ADMM-GP}. The first set of instances for the $k$-equipartition problem are described in~\cite{lisser2003graph},
the construction is as follows.
\begin{enumerate}	
    \item Choose edges of a complete graph randomly with probability $20\%$, $50\%$ and $80\%$.
    \item The nonzero edge weights are integers in the interval $(0,100]$. 
    \item Choose the partition numbers as divisors of the graph size $n$. 
\end{enumerate}
We name those three groups of instances rand20, rand50 and rand80, respectively.

\rew{Furthermore}, we consider instances that have been used in~\cite{armbruster2007}.
These are constructed in the following way.
\begin{itemize}
	\item $G_{|V |,|V |_p}$: Graphs $G(V,E)$, with $|V | \in \{124, 250, 500, 1000\}$ and four individual edge probabilities $p$. These
	probabilities were chosen depending on $|V|$, so that the average expected degree of each node was
	approximately $|V |_p = 2.5, 5, 10, 20$ \cite{johnson1989optimization}.
	\item $U_{|V |,|V |_{\pi d^2}}$: For a graph $G(V,E)$, first choose $2|V |$ independent numbers uniformly
	from the interval $(0, 1)$ and view them as coordinates of $|V |$ nodes on the unit square. Then, an edge
	is inserted between two vertices if and only if their Euclidian distance is less or equal
	to some pre-specified value $d$ \cite{johnson1989optimization}. Here $|V | \in \{ 500, 1000\}$ and $|V |_{\pi d^2} \in \{ 5, 10, 20, 40\}$.
	\item $mesh$: Instances from finite element meshes; all edge weights are equal to one~\cite{de1993graph}.
\end{itemize}

For GPKC we generate instances as described in~\cite{nguyen2016contributions}.
This is done by the following steps.
\begin{enumerate}
    \item Generate a random matrix with $20\%$, $50\%$ and $80\%$ of nonzeroes edge weights between $0$ and $100$, as vertex weights choose integers from the interval $(0,1000]$. 
    \item Determine a feasible solution for this instance for a $k$-equipartition problem by some heuristic method.
    \item Produce $1000$ permutations of the vertices in this $k$-equipartition.
    \item Calculate the capacity bound for each instance and select the one such that only $10\%$ of instances are feasible.
\end{enumerate}
We name those three groups of instances GPKCrand20, GPKCrand50 and GPKCrand80, respectively.

\subsection{Comparison of Post-processing Methods}\label{sec:numerical-post-proc}
Our first numerical comparisons evaluate the different post-processing methods used to produce safe lower bounds for the graph partition problems. Recall that in Section~\ref{sec:post-processing}, we introduced Algorithms~\ref{alg:lb-eADMM} and \ref{alg:post-process3}.
%

\begin{figure} 
	\centering
	\hspace*{\fill}
	\subfloat[$n=100,k=2$\label{fig:lb_a}]{\includegraphics[width = 2.2in]{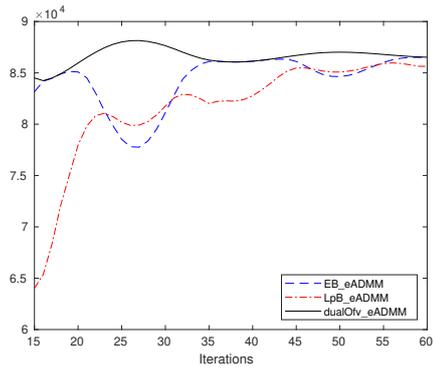}} \centerhfill
	\subfloat[$n=100,k=5$\label{fig:lb_b}]{\includegraphics[width = 2.2in]{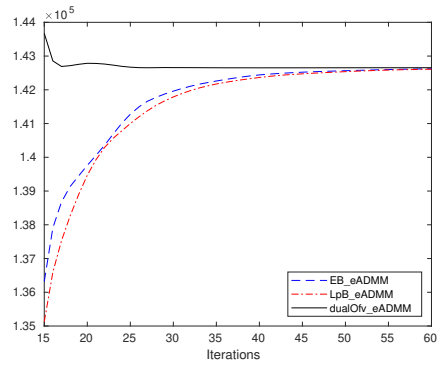}} \hspace*{\fill}
	
	\hspace*{\fill}
	\subfloat[$n=100,k=10$\label{fig:lb_c}]{\includegraphics[width = 2.2in]{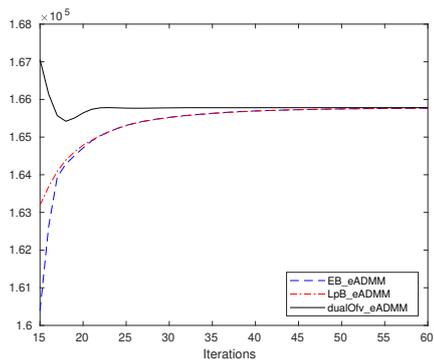}}\centerhfill
	\subfloat[$n=100,k=20$\label{fig:lb_d}]{\includegraphics[width = 2.2in]{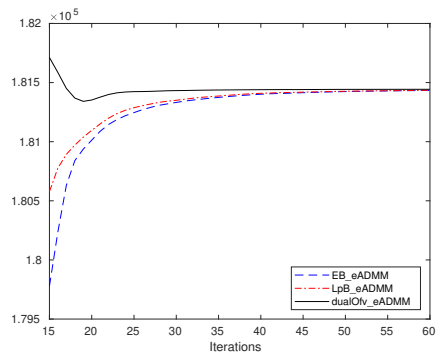}}\hspace*{\fill}
	\caption{Lower bounds obtained with post processing}
	\label{fig:lb_compare}
\end{figure}

Figure~\ref{fig:lb_compare} shows how the lower bounds from the post-processing methods evolve as the number of iterations of the extended ADMM increases.
\rew{We used the DNN relaxation on an instance of the $k$-equipartition problem of size $n=100$ and $k=2$.}
There are three lines: EB\_eADMM represents the lower bounds obtained by the rigorous lower bound method given in Algorithm~\ref{alg:post-process3}, LpB\_eADMM represents the linear programming bound given in Algorithm~\ref{alg:lb-eADMM} and dualOfv\_eADMM displays the approximate dual objective function value obtained by our extended ADMM. Figure~\ref{fig:lb_a} shows that the rigorous error bound method gives tighter bounds in general, while the linear programming bound method is more stable and less affected by the quality of the dual objective function value. The other figures indicate that for small $k$, the rigorous error bound method gives tighter bounds (see Figure \ref{fig:lb_b}), but as $k$ increases, the linear programming bound method dominates (see Figure~\ref{fig:lb_c} and~\ref{fig:lb_d}).
\begin{remark}
  We choose Algorithm~\ref{alg:lb-eADMM} in all following experiments as post-processing for $k$-equipartition problems because this method is more stable for varying~$k$.
  For GPKC we use Algorithm~\ref{alg:post-process3} for the post-processing since we have no information on the eigenvalue of an optimal solution.
\end{remark}

\subsection{Results for $k$-equipartition}

\subsubsection{Comparison of the Lower Bounds using SDP, DNN and Transitivity Constraints}\label{sec:lowerboundscomp-k-equi}
In this section we want to highlight the improvement of the bounds obtained from the relaxations introduced in Section~\ref{sec:gpp}.
Note that the timings for computing these bounds are discussed later in Section~\ref{sec:comptimes-k-equi}.

In practice, adding all the transitivity constraints is computationally too expensive, we run a DNN-based loop instead.
The idea is as follows. 
\begin{enumerate}
	\item Solve the DNN~\eqref{eq:keq-dnn} to obtain the solution $X^{DNN}$.
	\item Add $m_{met}$ transitivity constraints that are most violated by $X^{DNN}$ to the relaxation.
	\item  Solve the resulting relaxation and repeat adding newly violated constraints until the maximum number of iterations is reached or no more violated constraints are found.
\end{enumerate}

Tables~\ref{tab:keq_lb_80}, \ref{tab:keq_lb_50} and~\ref{tab:keq_lb_20} compare the lower bounds obtained from the relaxations for the $k$-equipartition problem. The improvements are calculated as $(d_{DNN} -d_{SDP})/d_{SDP}$ and $(d_{DNN+MET} -d_{SDP})/d_{SDP}$, respectively;  a `$-$' indicates that no transitivity constraints violated by the SDP solution of problem~\eqref{eq:keq-dnn} have been found.
In~\cite{rendl1999semidefinite} it has been observed that the
violation of the transitivity constraints is small and the
nonnegativity constraints $X \geq 0$ are more important than $X \in
\textrm{MET}$ when the partition number $k$ increases.
\rew{In our experiments we also observe that the improvement due to the
  nonnegativity constraints gets even better as $k$ increases.}

\input{randcsv_80}

	\input{randcsv_50}
	\input{randcsv_20}

\subsubsection{Comparisons between extended ADMM and Interior Point Methods (IPMs) on $k$-equipartition}\label{sec:comptimes-k-equi}
In this section we want to demonstrate the advantage of our extended ADMM over interior point methods.
For our comparisons we use Mosek~\cite{mosek}, one of the currently best performing \rew{interior point solvers}.

\rew{Note that computing an equipartition for these graphs using commercial
  solvers is out of reach. For instance, Gurobi obtains for a graph
  with $n=100$ vertices and $k\in \{2,4,5,10,20,25\}$ after
  120~seconds a gap of at least 80~\%, whereas we obtain gap
  up to at most 7~\%.}

We list the results for solving the SDP~\eqref{eq:keq-sdp} using an ADMM, and the results when solving the DNN relaxation~\eqref{eq:keq-dnn} by our extended ADMM and by Mosek.
We run the experiments on randomly generated graphs with 80\% density, the results are given in Table~\ref{tab:time_keq}. 

\input{randtimecsv}

Table~\ref{tab:time_keq} shows that the convergence behavior of the extended ADMM is not worse than the 2-block ADMM for SDP problems, and we can get a tighter lower bound by forcing nonnegativity constraints in the model without higher computational expense.

The results for Mosek solving problem~\eqref{eq:keq-dnn} clearly show that these problems are out of reach for interior point solvers. A ``$-$'' indicates that Mosek failed to solve this instance due to memory requirements.

\subsubsection{Heuristics on $k$-equipartition Problems}

We now compare the heuristics introduced in Section~\ref{sec:up_bound} to get upper bounds for the graph partition problems. 

\rew{We use the solutions obtained from the DNN relaxation to build upper bounds for the $k$-equipartition problem since the experimental results in Section~\ref{sec:lowerboundscomp-gpkc} showed that the DNN  relaxation has a good tradeoff between quality of the bound and solution time.}

We compare to the best know primal solution given in~\cite{armbruster2007}; we set the time limit for our heuristics to 5~seconds. The gaps between the upper bounds by Vc+2opt (resp. Hyp+2opt) and the best know solution are shown in Table~\ref{tab:keq_armbruster}. The primal bounds that are proved to be optimal are marked with ``$^*$''.

\input{armbruster_kequi}

Table~\ref{tab:keq_armbruster} shows that on small instances our heuristics can find upper bounds not worse than the best known upper bounds. For large instances, Vc+2opt performs  better than Hyp+2opt. The corresponding upper bounds are less than 10 \% away from the best known upper bounds, some of them computed using 5 hours.

\bigskip

We next compare the upper bounds for the instances rand80, rand50, and rand20.
Figure~\ref{fig:up_compare_100} shows that, for small instances (i.e., $n=100$), our hybrid methods (eg. Vc+2opt and Hyp+2opt) can find tight upper bounds quickly while simulated annealing (SA) needs a longer burning down time to achieve an upper bound of good quality. 

Figure~\ref{fig:up_compare_1000} shows how the heuristics behave for large-scale instances (i.e, $n= 1000$). The time limit is set to 5~seconds. Compared to Figure~\ref{fig:up_compare_100}, Vc+2opt and Hyp+2opt take more time to generate the first upper bounds but these upper bounds are much tighter than the one found by SA. Also, when the time limit is reached, the upper bounds found by Vc+2-opt and Hyp+2opt are much tighter than those from SA.

%
%

\begin{figure} 
	\centering
	\hspace*{\fill}%
	\subfloat[$n=100,k=2$]{\includegraphics[width = 2.2in]{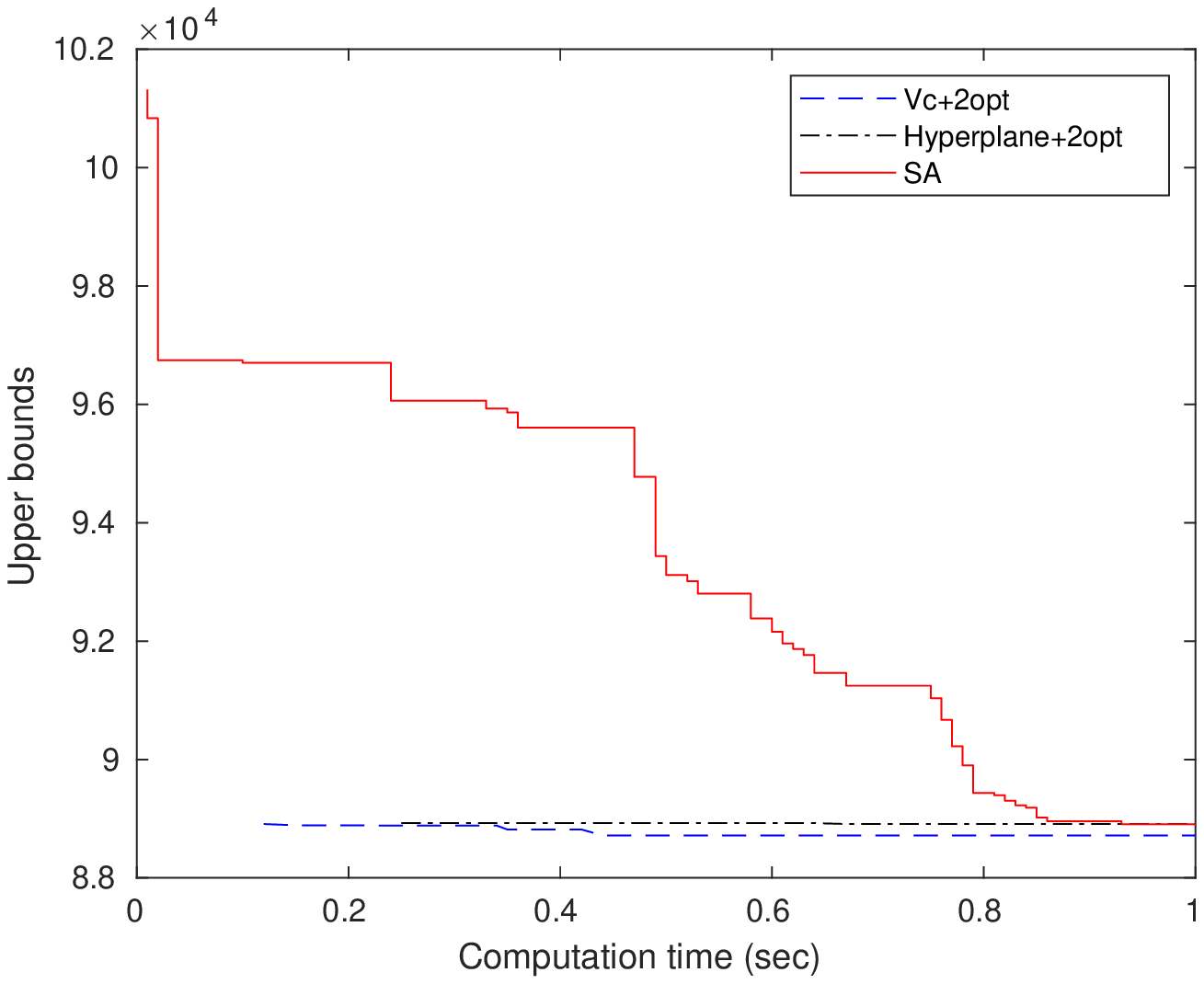}}\centerhfill
	\subfloat[$n=100,k=5$]{\includegraphics[width = 2.2in]{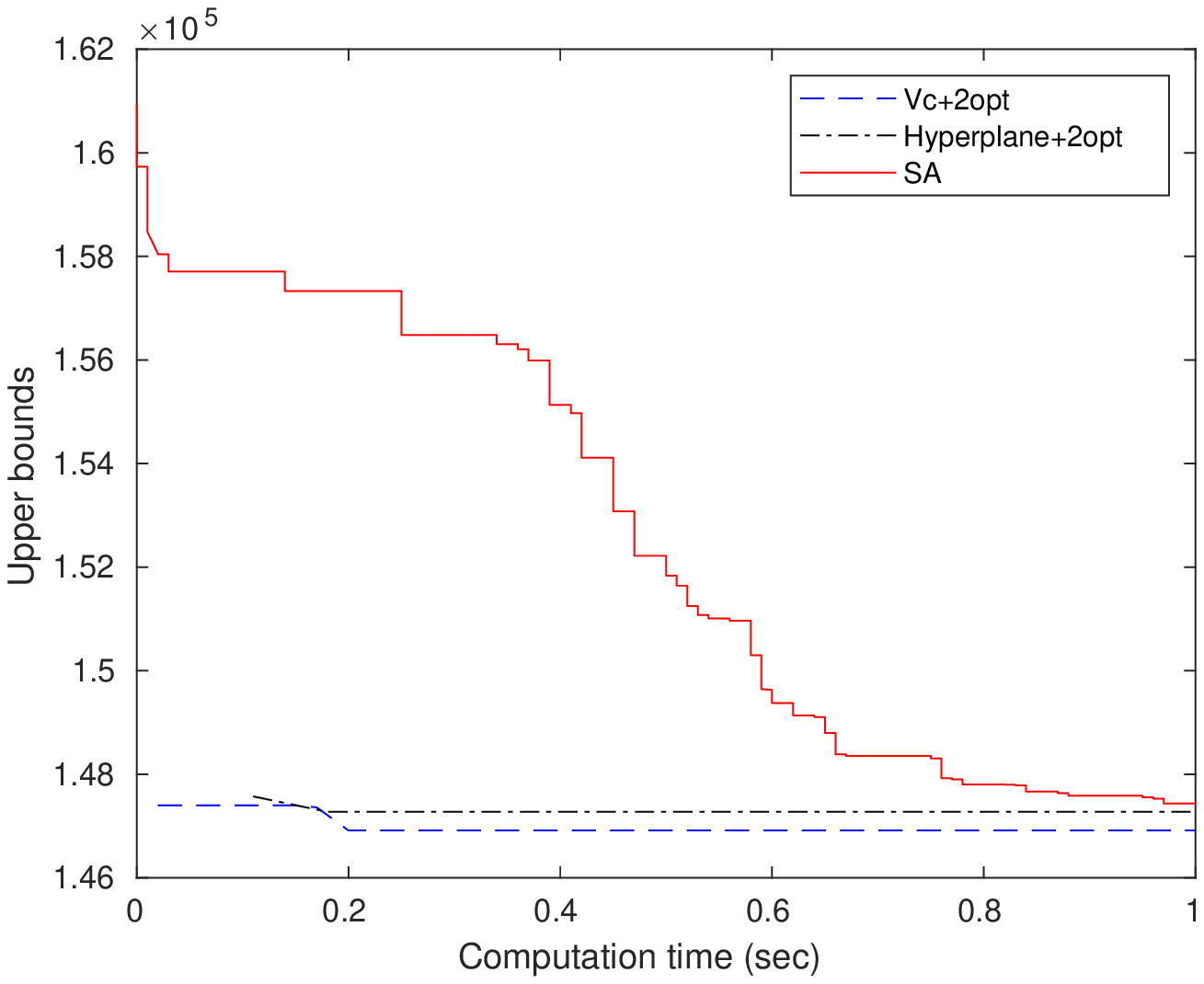}}\hspace*{\fill}
	
	\hspace*{\fill}%
	\subfloat[$n=100,k=10$]{\includegraphics[width = 2.2in]{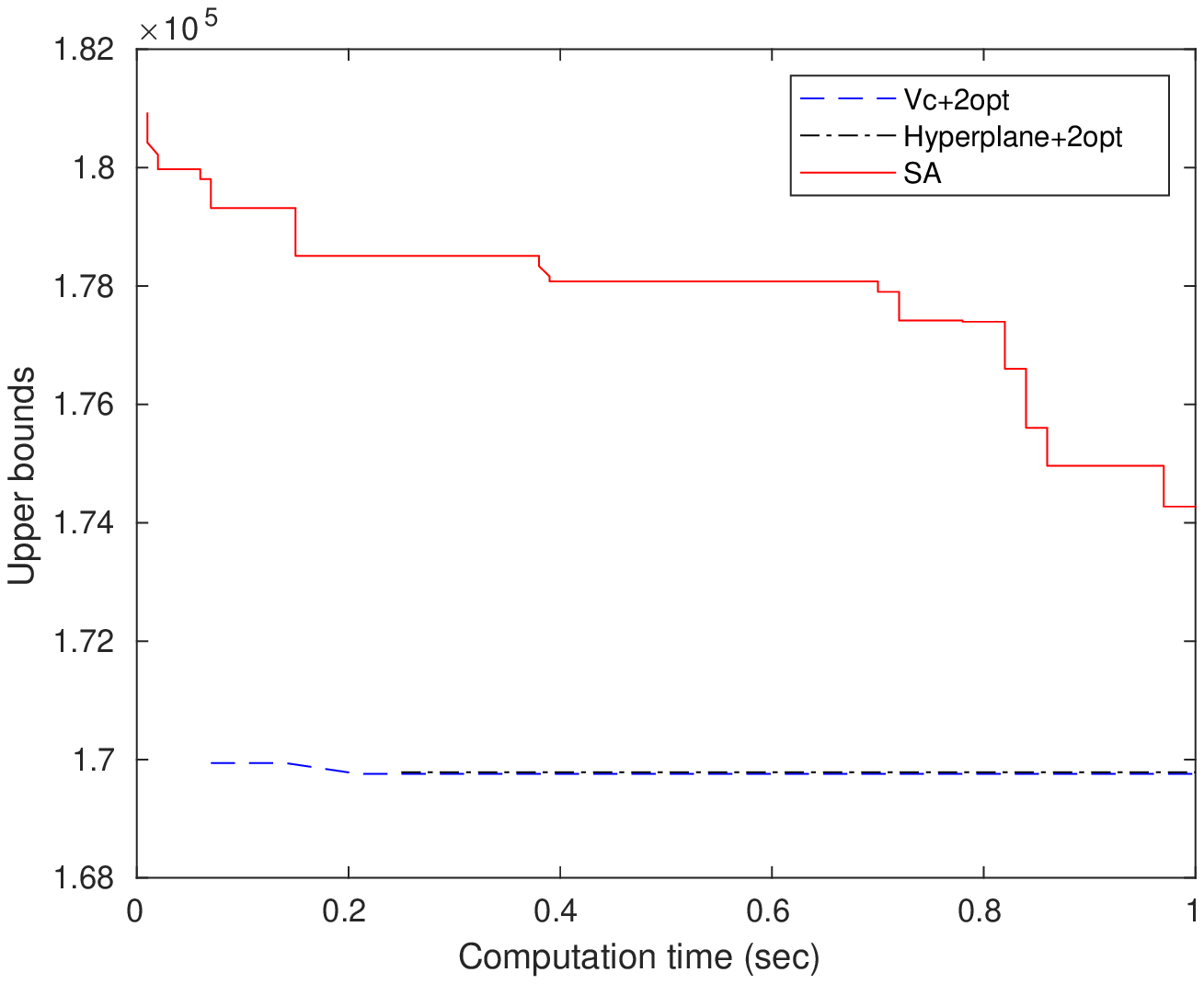}}\centerhfill
	\subfloat[$n=100,k=20$]{\includegraphics[width = 2.2in]{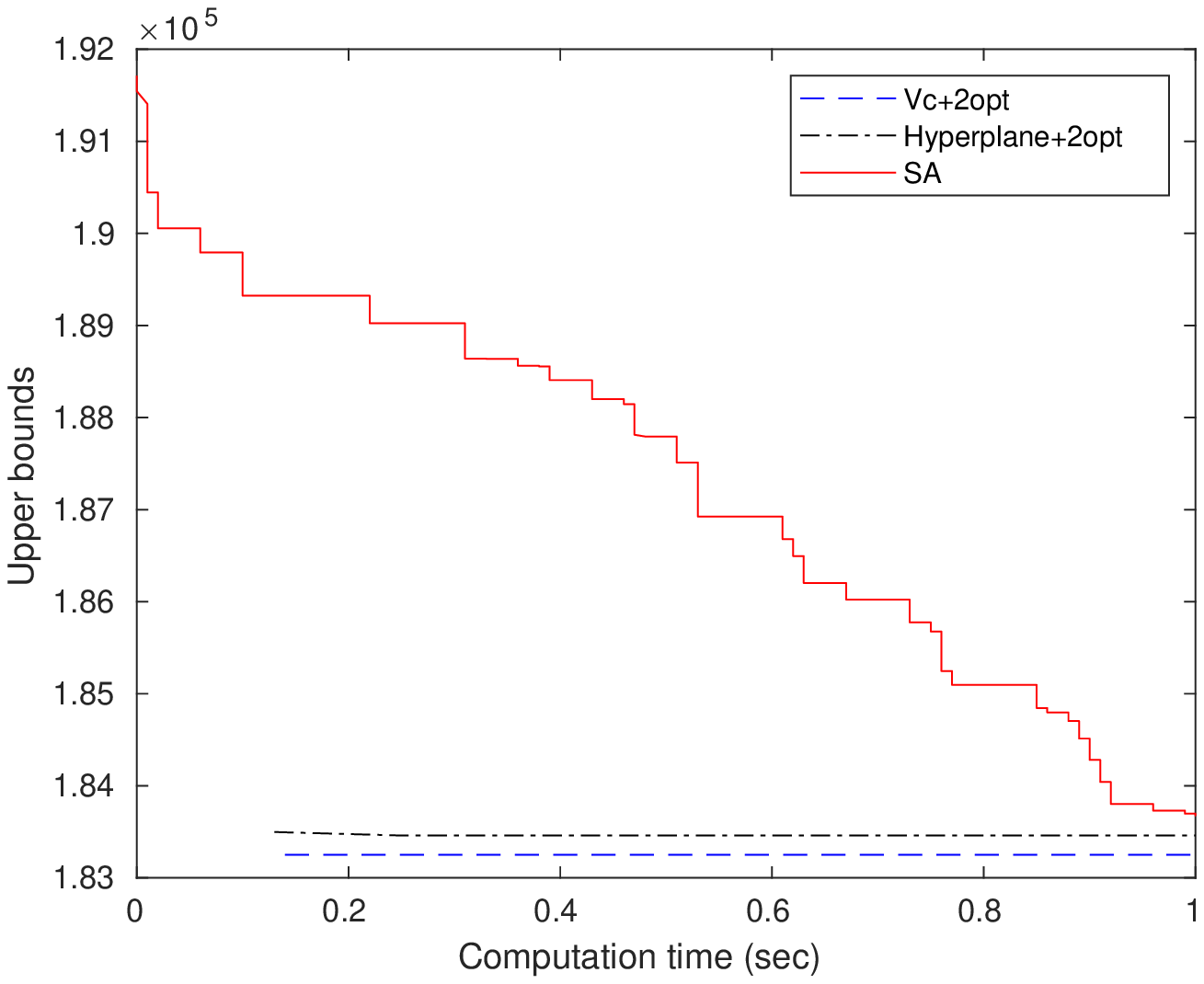}}\hspace*{\fill}%
	
	\caption{Upper bounds for $k$-equipartition problems on rand80 with $n=100$}\label{fig:up_compare_100}
\end{figure}

%

\begin{figure} 
	\centering
	\hspace*{\fill}%
	\subfloat[$n=1000,k=2$]{\includegraphics[width = 2.2in]{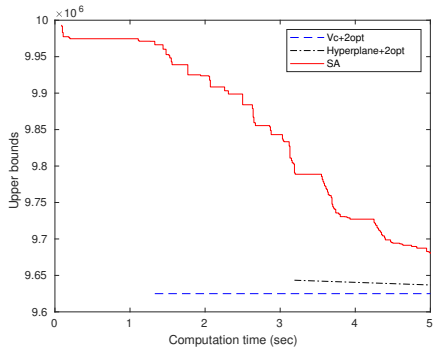}} \centerhfill
	\subfloat[$n=1000,k=10$]{\includegraphics[width = 2.2in]{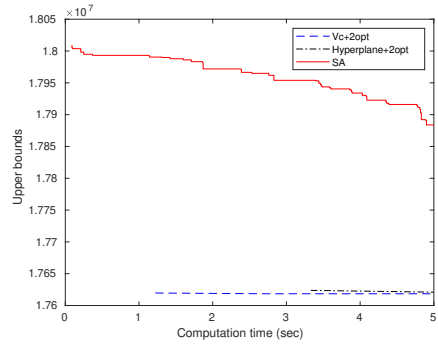}}\hspace*{\fill}%
	
	\hspace*{\fill}%
	\subfloat[$n=1000,k=40$]{\includegraphics[width = 2.2in]{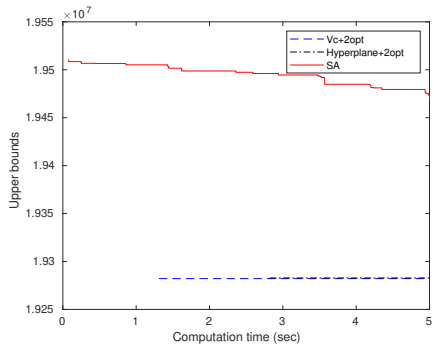}}\centerhfill
	\subfloat[$n=1000,k=100$]{\includegraphics[width = 2.2in]{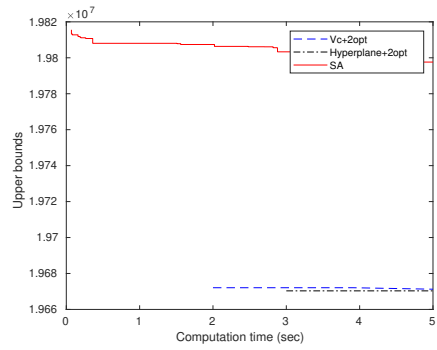}} \hspace*{\fill}%

	\caption{Upper bounds for $k$-equipartition problems on rand80 with $n=1000$}
	\label{fig:up_compare_1000}
\end{figure}

Tables~\ref{tab:keq_ub}, \ref{tab:keq_ub_50} and~\ref{tab:keq_ub_20} give a detailed comparison of the upper bounds for the instances rand80, rand50, and rand20, respectively.
We display the gap between the lower bounds obtained from the DNN relaxation~\eqref{eq:keq-dnn} and the upper bounds built by varied heuristics.
The time limit for the heuristics is set to 1~second for $n\in \{100,200\}$ and 3~seconds for $n\in \{900,1000\}$ for rand80.
For rand50 and rand20 we set the limit to 5~seconds.
The best upper bounds are typeset in bold.

The numbers confirm that Vc+2opt and Hyp+2opt can build tighter upper bounds than SA, in particular for the dense graphs rand80. Overall, Vc+2opt has the best performance.

Comparing lower and upper bounds, the numerical results show that our methods perform very well on dense graphs; for rand80 the largest gap is less than 4\%, for rand50 the largest gap is less than 6\%. As the randomly generated graph gets sparser, the gap between lower bounds and upper bounds increases, for rand20 the gap is bounded by 12\%.

	\input{kequpcsv}

\input{kequpcsv_50}

	\input{kequpcsv_20}

\bigskip

Comparing Tables~\ref{tab:keq_ub}, \ref{tab:keq_ub_50} and~\ref{tab:keq_ub_20}, it can be observed that the gaps get larger as the graph gets sparser.
\rew{We conjecture that} this is due to less tightness of the lower bound which \rew{is supported by} the following experiment.

We regard the best upper bounds obtained from all three heuristics within an increased time limit of 10~seconds. In this way, we should have an upper bound that approximates the optimal solution well enough for all densities.
As an example, in Table~\ref{tab:keq_rand_100_5} we report for a graph on 100 vertices and three different densities the lower and these upper bounds. We can clearly see that when the graph gets sparser, adding the nonnegativity constraints to the SDP relaxation~\eqref{eq:keq-sdp} gains more improvement. However, the gap between lower and upper bound gets worse as the graph gets sparser.

\input{rand_comp_100_5}

\subsection{Results for GPKC}
\subsubsection{Comparison of the Lower Bounds using SDP, DNN and Transitivity Constraints}\label{sec:lowerboundscomp-gpkc}

We now turn our attention to the GPKC problem.
We run experiments similar to those presented in Section~\ref{sec:lowerboundscomp-k-equi}, i.e., 
we solve DNN~\eqref{eq:gpkc-dnn} to obtain the solution $X^{DNN}$.
Table~\ref{tab:gpkc-lb} shows the lower bounds for GPKC problems on the randomly generated graphs rand80.
The improvements are calculated in the same way as in the previous section. The experimental results on GPKCrand50 and GPKCrand20 are omitted since they have a similar behavior.

The lower bounds obtained from different SDP relaxations show that when the capacity bound $W$ decreases (and thus the number of groups increases), the improvement of the nonnegativity constraints gets more significant. This is in line with the results for $k$-equipartition. 
And, also similar to $k$-equipartition, for GPKC the improvement due to the transitivity constraints is only minor.

\input{GPKCrandcsv}

\subsubsection{Comparisons between extended ADMM and IPMs for GPKC}\label{sec:comptimes-gpkc}

Table~\ref{tab:gpkc_time} compares the computation times when solving the DNN relaxations for the GPKC~\eqref{eq:gpkc-dnn} by the extended ADMM and Mosek, respectively. A ``--'' indicates  for extended ADMM that the maximum number of iterations is reached, and for Mosek that the instance could not be solved due to memory requirements. The results of the SDP relaxation~\eqref{eq:gpkc-sdp} in Table~\ref{tab:gpkc-lb} are computed using Mosek, hence we omit these timings in Table~\ref{tab:gpkc_time}.

\rew{In the thesis~\cite{nguyen2016contributions}, numerical results on the
GPKC are presented using an LP relaxation.
However, the method therein is capable of getting bounds either for very sparse
graphs of density at most 6~\% (up to 2000~vertices) or for graphs with up to 140~vertices and density of at most 50~\%. We clearly outperform these results in terms of the density of the graphs that can be considered.}

\input{gpkctimecsv}

While for instances of size $n=100$, the timings of the extended ADMM and Mosek are comparable, the picture rapidly changes as $n$ increases. For $n\ge 300$, Mosek cannot solve any instance while the extended ADMM manages to obtain bounds for instances with $n=500$ within one hour.

\subsubsection{Heuristics on GPKC problems}
As mentioned in Section~\ref{sec:up_bound}, the simulated annealing heuristic for the QAP cannot be applied to the GPKC, because there is no equivalence between the GPKC and the QAP. Therefore, we compare the upper bounds for the GPKC from the heuristic introduced in Section~\ref{sec:ub_gpkc} with the lower bounds given by the DNN relaxation~\eqref{eq:gpkc-dnn}. We set a time limit of 5~seconds.
 
Also, we set the maximum number of iterations to be 50\;000 for the sparse graph GPKCrand20, while the maximum numbers of iterations for GPKCrand50 and GPKCrand80 are 20\;000. In Table~\ref{tab:gpkc_ub}, \ref{tab:gpkc_ub_50} and \ref{tab:gpkc_ub_20},  a $^*$ indicates for the extended ADMM that the maximum number of iterations is reached.

Table~\ref{tab:gpkc_ub} shows that the gaps between the lower and upper bounds are less than 3\% for GPKCrand80, they are less than 7\% for GPKCrand50, see Table~\ref{tab:gpkc_ub_50}, and for GPKCrand20, the gaps are less than 15\%, see Table~\ref{tab:gpkc_ub_20}.
Similar to the $k$-equipartition problem, we note that computing the lower bound on the sparse instances is harder. The maximum number of iterations is reached for rand20 much more often than for rand80 or rand50.

	\input{GPKCrandupcsv}

	\input{GPKCrandupcsv_50}

	\input{GPKCrandupcsv_20}

\section{Conclusions}\label{sec:conclusion}
In this paper we first introduce different SDP relaxations for $k$-equipartition problems and GPKC problems. Our tightest SDP relaxations, problems \eqref{eq:keq-met} and \eqref{eq:gpkc-met}, contain all nonnegativity constraints and transitivity constraints, which bring $O(n^3)$ constraints in total. Another kind of tight SDP relaxation, \eqref{eq:keq-dnn} and \eqref{eq:gpkc-dnn}, has only nonnegativity constraints.
While it is straight forward to consider the constraint $X\ge 0$ in a 3-block ADMM, including all the transitivity constraints is impractical.
Therefore, our strategy is to solve \eqref{eq:keq-dnn} and \eqref{eq:gpkc-dnn} and then adding violated transitivity constraints in loops to tighten both SDP relaxations.

In order to deal with the SDP problems with inequality and bound constraints, we extend the classical 2-block ADMM, which only deals with equations, to the extended ADMM for general SDP problems. This algorithm is designed to solve large instances that interior point methods fail to solve. We also introduce heuristics that build upper bounds from the solutions of the SDP relaxations. The heuristics include two parts, first we round the SDP solutions to get a feasible solution for the original graph partition problem, then we apply 2-opt methods to locally improve this feasible solution. In the procedure of rounding SDP solutions, we introduce two algorithms, the vector clustering method and the generalized hyperplane rounding method. Both methods perform well with the 2-opt method.

The extended ADMM can solve general SDP problems efficiently. For SDP problems with bound constraints, the extended ADMM deals with them separately from inequalities and equations, thereby solving the problems more efficiently. Mosek fails to solve the DNN relaxations of problems with $n\geq 300$ due to memory requirements while the extended ADMM can solve the DNN relaxations for $k$-equipartition problems on large instances up to $n =1000$ within as few as 5 minutes and for GPKC problems up to $n=500$ within as little as 1 hour.

We run numerical tests on instances from the literature and on randomly generated graphs with different densities. The results show that SDP relaxations can produce tighter bounds for dense graphs than sparse graphs.
In general, the results show that \rew{nonnegativity constraints give more improvement when $k$ increases.}

We compare our heuristics with a simulated annealing method in the generation of upper bounds for $k$-equipartition problems. Our heuristics obtain upper bounds displaying better quality within a short time limit, especially for large instances. Our methods show better performance on dense graphs, where the final gaps are less than 4\% for graphs with 80\% density, while the gaps between lower and upper bounds for sparse graphs with 20\% density are bounded by 12\%. This is mainly due to the tighter lower bounds for dense graphs.

\ifSpringer
\else
\bibliographystyle{plainnat}

\subsection*{Acknowledgments}
	We thank Kim-Chuan Toh for bringing our attention to~\cite{jansson2007rigorous} and 
	for providing an implementation of the method therein.
	We would like to thank the reviewers for their thoughtful comments and efforts towards 
	improving this paper.
\fi

\bibliography{mybib}
\end{document}

%% file: randcsv_80.tex
\begin{filecontents*}{rand_keq_sdp_dnn.csv}
n,partition number,SDP,DNN,improvement_dnn,DNN+tri,improvement_dnnmet
100,2,86605.32415,86616.53588,0.01,87900.2,1.5
,4,129938.4478,132491.1448,1.96,132777,2.18
,5,138603.1268,142647.172,2.92,142736,2.98
,10,155933.1108,165781.0059,6.32,165782,6.32
,20,164598.4615,181444.0279,10.23,,
,25,166331.6478,185340.169,11.43,,
200,2,362347.8635,362366.8733,0.01,365512,0.87
,4,543624.41,550178.1363,1.21,550616,1.29
,5,579880.7901,590123.7556,1.77,590281,1.79
,10,652392.8083,677024.5917,3.78,677048,3.78
,20,688652.1546,730635.4291,6.1,730637,6.1
,40,706783.9823,766592.4914,8.46,,
900,2,7667181.824,7667181.824,0,7685601.679,0.24
,5,12277076.08,12326530.46,0.4,12331888.86,0.45
,10,13813797.53,13957505.42,1.04,,
,20,14582084.03,14854713.5,1.87,,
,30,14838190.47,15197038.17,2.42,,
,50,15043065.6,15518811.79,3.16,,
,100,15196404.75,15821635.34,4.11,,
,300,15299130.76,16068306.79,5.03,,
1000,2,9512467.498,9520746.692,0.09,9534980.438,0.24
,5,15233252.24,15287792.55,0.36,15295612.25,0.41
,10,17139974.48,17302696.39,0.95,,
,20,18093221.41,18404248.61,1.72,,
,40,18569845.83,19056214.07,2.62,,
,50,18665166.6,19211689.51,2.93,,
,100,18855894.81,19578435.97,3.83,,
,200,18951231.87,19800076.6,4.48,,
\end{filecontents*}
\pgfplotstabletypeset[  
col sep=comma, 
fixed zerofill,
fixed,
font=\footnotesize,
multicolumn names, 
display columns/0/.style={int detect,
	column name=$n$, 
	},  
display columns/1/.style={int detect,
	column name=$k$,
	},
display columns/2/.style={,
	column name=$lb_{SDP}$,
	},
display columns/3/.style={
	column name=$lb_{DNN}$,
	},
display columns/4/.style={
	column name=Imp~\%,
	},
display columns/5/.style={empty cells with={--},
	column name=$lb_{DNN+MET}$,
},
display columns/6/.style={
	column name=Imp~\%,empty cells with={--},
},
every row no 6/.style={before row={\midrule}},
every row no 11/.style={after row={\midrule
$\vdots$	&$\vdots$  &$\vdots$  &$\vdots$ &$\vdots$ &$\vdots$ &$\vdots$ \\}},
every row no 12/.style={before row={\midrule}},
every row no 20/.style={before row={\midrule}},
every head row/.style={
	fixed,precision=2,
	before row={
		\caption{$k$-equipartitioning lower bounds on rand80}\label{tab:keq_lb_80}\\
		\toprule}, 
	after row={
\endfirsthead
		\midrule} 
},
every last row/.style={after row=\bottomrule}, 
]{rand_keq_sdp_dnn.csv} 

%% file: randcsv_50.tex
\begin{filecontents*}{rand_keq_sdp_dnn_50.csv}
n,partition number,SDP,DNN,improvement_dnn,DNN+tri,improvement_dnnmet
100,2,47928.96915,47928.23129,0,49300.25814,2.86
,4,71902.51497,74349.77792,3.4,74714.81417,3.91
,5,76697.24448,80404.43181,4.83,80566.86086,5.05
,10,86286.59349,94925.78418,10.01,94931.81917,10.02
,20,91081.3637,105704.483,16.06,105705.399,16.06
,25,105704.483,108816.5508,2.94,108817.9707,2.95
200,2,209582.3511,209592.4004,0,212563.4873,1.42
,4,314422.5824,320218.8717,1.84,320740.0136,2.01
,5,335391.7808,344366.4606,2.68,344545.587,2.73
,10,377332.4617,398653.0504,5.65,398686.115,5.66
,20,398306.21,434293.1575,9.03,,
,40,408794.1891,462004.0853,13.02,,
900,2,4641598.866,4644349.583,0.06,4657296.91,0.34
,5,7431923.288,7481538.623,0.67,7484865.681,0.71
,10,8362037.006,8499452.661,1.64,,
,20,8827102.407,9078072.186,2.84,,
,30,8982122.954,9308248.743,3.63,,
,50,9106135.447,9533366.094,4.69,,
,100,9199147.869,9776472.27,6.28,,
,300,9261139.257,10009909.74,8.09,,
1000,2,5760088,5760088,0,5778737.722,0.32
,5,9224420.72,9278976.84,0.59,9283474.065,0.64
,10,10378975.72,10534126.08,1.49,,
,20,10956261.72,11242920.84,2.62,,
,40,11244891.69,11684089.99,3.91,,
,50,11302617.95,11794006.69,4.35,,
,100,11418071.85,12084361.06,5.84,,
,200,11475814.32,12289919.2,7.09,,
\end{filecontents*}
\pgfplotstabletypeset[  
multicolumn names, 
col sep=comma, 
fixed zerofill,
fixed,
font=\footnotesize,
display columns/0/.style={int detect,
	column name=$n$, 
	},  
display columns/1/.style={int detect,
	column name=$k$,
	},
display columns/2/.style={
	column name=$lb_{SDP}$,
	},
display columns/3/.style={
	column name=$lb_{DNN}$,
	},
display columns/4/.style={
	column name=Imp~\%,
	},
display columns/5/.style={empty cells with={--},
	column name=$lb_{DNN+MET}$,
},
display columns/6/.style={empty cells with={--},
	column name=Imp~\%,
},
every row no 6/.style={before row={\midrule}},
every row no 11/.style={after row={\midrule
		$\vdots$	&$\vdots$  &$\vdots$  &$\vdots$ &$\vdots$ &$\vdots$ &$\vdots$ \\}},
every row no 12/.style={before row={\midrule}},
every row no 20/.style={before row={\midrule}},
every head row/.style={
	fixed,precision=2,
	before row={
		\caption{$k$-equipartitioning lower bounds on rand50}\label{tab:keq_lb_50}\\
		\toprule}, 
	after row={
\endfirsthead
		\midrule} 
},
every last row/.style={after row=\bottomrule}, 
]{rand_keq_sdp_dnn_50.csv} 

%% file: randcsv_20.tex
\begin{filecontents*}{rand_keq_sdp_dnn_20.csv}
n,partition number,SDP,DNN,improvement_dnn,DNN+tri,improvement_dnnmet
100,2,14747.59239,14747.59239,0,15762.93125,6.88
,4,22127.03738,23466.70578,6.05,23980.79707,8.38
,5,23602.89189,25695.45431,8.87,26002.66811,10.17
,10,26554.53862,31618.12337,19.07,31684.76963,19.32
,20,28030.29919,36793.38904,31.26,36803.49526,31.3
,25,28325.4308,38432.17038,35.68,38436.72355,35.7
200,2,70569.62219,70569.62219,0,72897.33435,3.3
,4,105873.9024,109373.1676,3.31,110128.5907,4.02
,5,112935.883,118429.9547,4.86,118775.6827,5.17
,10,127058.5202,140350.6291,10.46,140393.1352,10.49
,20,134120.2507,156812.3976,16.92,156815.8867,16.92
,40,137652.3619,170920.8351,24.17,170921.0574,24.17
900,2,1715687.426,1715987.611,0.02,1726605.374,0.64
,5,2747123.67,2782017.73,1.27,2783340.636,1.32
,10,3090931.896,3184081.522,3.01,,
,20,3262832.489,3430495.428,5.14,,
,30,3320134.283,3535568.528,6.49,,
,50,3365975.136,3644361.833,8.27,,
,100,3400355.41,3767595.302,10.8,,
,300,3423251.062,3945849.468,15.27,,
1000,2,2140975.918,2141688.491,0.03,2152422.294,0.53
,5,3427722.385,3465886.871,1.11,3467615.206,1.16
,10,3856783.946,3960838.714,2.7,,
,20,4071315.835,4260408.029,4.64,,
,40,4178581.289,4463023.577,6.81,,
,50,4200033.959,4516439.336,7.53,,
,100,4242940.303,4661565.65,9.87,,
,200,4264393.833,4801394.974,12.59,,
\end{filecontents*}
\pgfplotstabletypeset[  
multicolumn names, 
col sep=comma, 
fixed zerofill,
fixed,
	font=\footnotesize,
display columns/0/.style={int detect,
	column name=$n$, 
	},  
display columns/1/.style={int detect,
	column name=$k$,
	},
display columns/2/.style={
	column name=$lb_{SDP}$,
	},
display columns/3/.style={
	column name=$lb_{DNN}$,
	},
display columns/4/.style={
	column name=$\textrm{Imp}~\%$,
	},
display columns/5/.style={empty cells with={--},
	column name=$lb_{DNN+MET}$,
},
display columns/6/.style={empty cells with={--},
	column name=$\textrm{Imp}~\%$,
},
every row no 6/.style={before row={\midrule}},
every row no 11/.style={after row={\midrule
	$\vdots$		& $\vdots$ &$\vdots$  &$\vdots$ & $\vdots$& $\vdots$& $\vdots$\\}},
every row no 12/.style={before row={\midrule}},
every row no 20/.style={before row={\midrule}},
every head row/.style={
	fixed,precision=2,
	before row={
		\caption{$k$-equipartitioning lower bounds on rand20}\label{tab:keq_lb_20}\\
		\toprule}, 
	after row={
\endfirsthead
		\midrule} 
},
every last row/.style={after row=\bottomrule}, 
]{rand_keq_sdp_dnn_20.csv} 

%% file: randtimecsv.tex
	\begin{filecontents*}{admm_keq_rand_time.csv}
n,k,SDP iteration,Sdp CPUTIME,Dnn iterations,Dnn CPUTIME,laptop cputime
100,2,220,1.59,152,1.03,40.12
,4,302,2.16,79,0.48,41.13
,5,309,2.07,76,0.48,42.22
,10,291,2.17,88,0.81,37.97
,20,285,2.09,231,1.56,39.5
,25,286,1.81,209,1.42,35.31
200,2,223,3.85,162,3.71,2414.38
,4,339,5.92,81,1.43,2279.18
,5,361,5.63,75,1.5,2266.13
,10,382,6.05,69,1.24,2028
,20,369,5.75,104,2.05,1816.87
,40,379,6.01,324,7.05,2187.09
300,2,245,8.15,173,7.82,
,3,307,10.65,114,4.06,
,5,339,11.25,79,3.22,
,6,346,12.53,72,2.73,
,10,365,12.68,65,2.62,
,30,413,14.72,130,5.82,
,50,466,17.07,286,12.13,
,100,580,21.99,1158,50.96,
400,2,267,15.45,183,10.42,
,4,345,20.87,90,6.62,
,5,360,21.55,80,5.19,
,8,382,25.72,64,4.57,
,10,392,23.83,60,4.31,
,20,440,28.03,68,6.24,
,40,471,32.72,154,11.36,
,100,616,43.61,734,55.43,
500,2,256,20.87,187,16.14,
,5,430,34.47,80,7.39,
,10,485,39.49,59,5.06,
,20,535,44.57,62,6.28,
,25,546,47.44,71,6.58,
,50,565,46.12,183,16.99,
,100,683,54.59,559,51.17,
600,2,274,26.85,179,18.68,
,3,348,38.9,112,12.48,
,5,422,46.77,81,8.91,
,6,437,48.41,73,8.11,
,10,482,53.68,58,6.69,
,20,536,59.2,59,6.96,
,30,552,61.78,74,8.75,
,50,565,64.16,155,18.91,
,100,714,82.19,442,56.02,
,200,923,106.55,890,111.57,
700,2,266,39.92,209,35.69,
,5,452,70.57,80,12.74,
,10,535,85.25,57,9.14,
,20,590,103.52,57,9.53,
,35,612,97.37,75,13.13,
,50,620,99.19,132,21.64,
,70,658,104.88,222,37.7,
,100,776,124.72,380,64.98,
800,2,253,51.79,211,48.68,
,5,451,96.57,79,17.6,
,10,537,116.06,57,12.71,
,20,595,128.97,55,12.55,
,40,620,134.88,75,17.59,
,50,625,136.16,118,28.53,
,100,759,166.67,333,78.13,
,200,1087,240.91,776,186.33,
900,2,246,67.94,218,65.91,
,5,434,124.6,80,23.52,
,10,523,151.21,57,16.92,
,20,583,169.88,54,16.58,
,30,601,174.28,63,19.73,
,50,613,178.47,104,35.4,
,100,711,207.95,305,95,
,300,1644,489.24,515,162.4,
1000,2,246,88.62,200,77.91,
,5,449,166.91,80,30.85,
,10,549,206.41,57,21.88,
,20,612,232.1,52,20.45,
,40,637,249.78,72,30.02,
,50,642,241.78,91,38.88,
,100,705,265.93,278,114.78,
,200,1306,496.7,682,277.84,
\end{filecontents*}
\pgfplotstabletypeset[  
	font=\footnotesize,
	multicolumn names, 
	col sep=comma, 
	fixed zerofill,
	fixed,
	display columns/0/.style={int detect,
		column name=$n$, 
		},  
	display columns/1/.style={int detect,
		column name=$k$,
		},
	display columns/2/.style={int detect,
		column name=$\textrm{Iter}$,
		},
	display columns/3/.style={
		column name=$\textrm{CPU time}$,
		},
	display columns/4/.style={int detect,
		column name=$\textrm{Iter}$,
		},
	display columns/5/.style={
		column name=$\textrm{CPU time}$,
	},
	display columns/6/.style={
		column name=$\textrm{CPU time}$,empty cells with={--}, 
	},
	every row no 6/.style={before row={\midrule}},
	every row no 12/.style={before row={\midrule}},
	every row no 20/.style={before row={\midrule}},
	every row no 28/.style={before row={\midrule}},
	every row no 35/.style={before row={\midrule}},
	every row no 45/.style={before row={\midrule}},
	every row no 53/.style={before row={\midrule}},
	every row no 61/.style={before row={\midrule}},
	every row no 69/.style={before row={\midrule}},
	every head row/.style={fixed,precision=2,
		before row={
		\caption{Computation times for $k$-equipartitioning problems}\label{tab:time_keq}\\
		\toprule 
		& & \multicolumn{4}{c}{ADMM} & Mosek \\
		& & \multicolumn{2}{c}{SDP \eqref{eq:keq-sdp}} & \multicolumn{2}{c}{DNN \eqref{eq:keq-dnn}} & DNN \eqref{eq:keq-dnn} \\},
		after row={
		& & &(\si{\second}) & &(\si{\second}) &(\si{\second})  \\ 
	\endfirsthead
		\midrule 
	 } 
		},
	every last row/.style={after row=\bottomrule}, 
	]{admm_keq_rand_time.csv} 
	

%% file: armbruster_kequi.tex
\pgfplotstableread[col sep=comma]{armbruster_kequi_data.tex}\data
\pgfplotstabletypeset[  
	font=\footnotesize,
	multicolumn names, 
	columns={
	[index]0,[index]1,[index]2,[index]3,[index]4,[index]5,[index]6,[index]7
	},
	fixed zerofill,
	fixed,
	display columns/0/.style={string type,
	column name=Graph, 
	},  
	display columns/1/.style={int detect,string type,
	column name=$n$,
	},  
	display columns/2/.style={int detect,string type,
	column name=$k$,
	},
	display columns/3/.style={precision=0,
	column name=Vc+2opt,
	},
	display columns/4/.style={precision=2,
		column name=Gap,
	},
	display columns/5/.style={precision=0,
	column name=Hyp+2opt,
	},
	display columns/6/.style={precision=2,
	column name=Gap,
	},
	display columns/7/.style={precision=0,
		column name=Best Bound,
	},
	every row no 8/.style={before row={\midrule}},
	every row no 23/.style={after row={\midrule}},
	every head row/.style={
	before row={
		\caption{Feasible solutions for the graphs from~\cite{armbruster2007} (the time limit  is 5 seconds, optimal solutions are indicated by a ``$^*$'')}\label{tab:keq_armbruster}\\
		\toprule}, 
	after row={
		&  &  & &\si{\percent} & &\si{\percent} \\
		\endfirsthead
	  	\midrule} 
	},	
	row iteratorOptiter/.style={
	every row #1 column 7/.style={
		postproc cell content/.append style={
			/pgfplots/table/@cell content/.add={$^*$}}
	}
	},
	row iteratorOptiter/.list={0,1,4,5,8,9,10,11,12,24,25,26,27,28},
	every last row/.style={after row=\bottomrule}
]{\data} 

%% file: kequpcsv.tex
	\pgfplotstableread[col sep=comma]{admm_keq_rand_ub.csv}\data
	\pgfplotstabletypeset[  
	font=\footnotesize,
	multicolumn names, 
	fixed zerofill,
	fixed,
	display columns/0/.style={int detect,string type,
		column name=$n$, 
		},  
	display columns/1/.style={int detect,string type,
		column name=$k$,
		},
	display columns/2/.style={precision=2,
		column name=$lb_{DNN}$},
	display columns/3/.style={precision=0,
		column name=Vc+2opt},
	display columns/4/.style={
		column name=Gap},
	display columns/5/.style={precision=0,
		column name=Hyp+2opt,
		},
	display columns/6/.style={
		column name=Gap},
	display columns/7/.style={precision=0,
		column name=SA
	},
	display columns/8/.style={
	column name=Gap},
	every row no 6/.style={before row={\midrule}},
	every row no 11/.style={after row={\midrule
			$\vdots$	&$\vdots$  &$\vdots$  &$\vdots$ &$\vdots$ \\}},
	every row no 12/.style={before row={\midrule}},
	every row no 20/.style={before row={\midrule}},
	every head row/.style={
		before row={
			\caption{Feasible solutions for randomly generated graphs rand80 (for instances with $n \in \{100,200\}$, the time limit is 1 second; for instances with $n \in \{ 900, 1000\}$, the limit is 3 seconds)}\label{tab:keq_ub}\\
			\toprule}, 
		after row={
 &  &  & &\si{\percent} &  &\si{\percent} & &\si{\percent}\\
		  	\midrule} 
	},	
	row iteratorNC/.style={
		every row #1 column 3/.style={
		postproc cell content/.append style={
			/pgfplots/table/@cell content/.add={$\bf}{$}}}
	},
	row iteratorHy/.style={
		every row #1 column 5/.style={
			postproc cell content/.append style={
				/pgfplots/table/@cell content/.add={$\bf}{$}}}
	},	
	every last row/.style={after row=\bottomrule}, 
	row iteratorNC/.list={0,2,3,5,6,7,9,10,11,12,13,14,19,20,21,22,23,24,25},
	row iteratorHy/.list={1,4,8,15,16,17,18,26,27},
	]{\data} 

%% file: kequpcsv_50.tex
	\pgfplotstableread[col sep=comma]{admm_keq_rand_ub_50_5s.csv}\data
	\pgfplotstabletypeset[  
	font=\footnotesize,
	fixed zerofill,
	fixed,
	multicolumn names, 
	display columns/0/.style={int detect,string type,
	column name=$n$, 
	},  
	display columns/1/.style={int detect,string type,
		column name=$k$,
	},
	display columns/2/.style={precision=2,
		column name=$lb_{DNN}$},
	display columns/3/.style={precision=0,
		column name=Vc+2opt},
	display columns/4/.style={
		column name=Gap},
	display columns/5/.style={precision=0,
		column name=Hyp+2opt,
	},
	display columns/6/.style={
		column name=Gap},
	display columns/7/.style={precision=0,
		column name=SA
	},
	display columns/8/.style={
		column name=Gap
	},
	every row no 6/.style={before row={\midrule}},
	every row no 11/.style={after row={\midrule
			$\vdots$	&$\vdots$  &$\vdots$  &$\vdots$ &$\vdots$ \\}},
	every row no 12/.style={before row={\midrule}},
	every row no 18/.style={before row={\midrule}},
	every head row/.style={
		before row={
			\caption{Feasible solutions for randomly generated graphs rand50 (time limit 5 seconds)}\label{tab:keq_ub_50}\\
			\midrule}, 
		after row={
 &  &  & &\si{\percent} &  &\si{\percent} & &\si{\percent}  \\
		  	\endfirsthead
	  			\midrule } 
	},	
	row iteratorNC/.style={
		every row #1 column 3/.style={
		postproc cell content/.append style={
			/pgfplots/table/@cell content/.add={$\bf}{$}}}
	},
	row iteratorHy/.style={
		every row #1 column 5/.style={
			postproc cell content/.append style={
				/pgfplots/table/@cell content/.add={$\bf}{$}}}
	},
	row iteratorSA/.style={
		every row #1 column 7/.style={
		postproc cell content/.append style={
			/pgfplots/table/@cell content/.add={$\bf}{$}}}
	},	
	every last row/.style={after row=\bottomrule}, 
	row iteratorNC/.list={0,1,2,6,7,8,12,13,14,15,16,17,18,19,20,22,23,24,25,27},
	row iteratorHy/.list={0,3,4,21,26},
	row iteratorSA/.list={5,9,10,11},
	]{\data} 

%% file: kequpcsv_20.tex
	\pgfplotstableread[col sep=comma]{admm_keq_rand_ub_20_5s.csv}\data
	\pgfplotstabletypeset[  
	font=\footnotesize,
	fixed zerofill,
	fixed,
	multicolumn names, 
	display columns/0/.style={int detect,string type,
		column name=$n$, 
	},  
	display columns/1/.style={int detect,string type,
		column name=$k$,
	},
	display columns/2/.style={precision=2,
		column name=$lb_{DNN}$},
	display columns/3/.style={precision=0,
		column name=Vc+2opt},
	display columns/4/.style={
		column name=Gap},
	display columns/5/.style={precision=0,
		column name=Hyp+2opt,
	},
	display columns/6/.style={
		column name=Gap},
	display columns/7/.style={precision=0,
		column name=SA
	},
	display columns/8/.style={
		column name=Gap
	},
	every row no 6/.style={before row={\midrule}},
	every row no 11/.style={after row={\midrule
			$\vdots$	&$\vdots$  &$\vdots$  &$\vdots$ &$\vdots$ \\}},
	every row no 12/.style={before row={\midrule}},
	every row no 18/.style={before row={\midrule}},
	every head row/.style={
		before row={
			\caption{Feasible solutions for randomly generated graphs rand20 (time limit 5 seconds) }\label{tab:keq_ub_20}\\
			\midrule}, 
		after row={
 &  &  & &\si{\percent} &  &\si{\percent} & &\si{\percent}  \\
  \endfirsthead
	\midrule} 
	},	
	row iteratorNC/.style={
		every row #1 column 3/.style={
		postproc cell content/.append style={
			/pgfplots/table/@cell content/.add={$\bf}{$}}}
	},
	row iteratorHy/.style={
		every row #1 column 5/.style={
			postproc cell content/.append style={
				/pgfplots/table/@cell content/.add={$\bf}{$}}}
	},	
	row iteratorSA/.style={
		every row #1 column 7/.style={
			postproc cell content/.append style={
				/pgfplots/table/@cell content/.add={$\bf}{$}}}
	},	
	every last row/.style={after row=\bottomrule}, 
	row iteratorNC/.list={0,1,3,4,8,9,11,12,13,14,15,16,17,18,19,21,22,23,24,25,26},
	row iteratorHy/.list={0,6,7,20,27},
	row iteratorSA/.list={2,5,10},
	]{\data} 

%% file: rand_comp_100_5.tex
	\begin{filecontents*}{admm_keq_rand_compare_n100_k5.csv}
p,SDPlb,DNNlb,improvement,best ub(10s),gap
80,138603.1268,142647.172,2.92,146565,2.75
50,76697.24448,80404.43181,4.83,84334,4.89
20,23602.89189,25696.89661,8.87,28320,10.21
\end{filecontents*}
\pgfplotstabletypeset[  
	font=\footnotesize,
	multicolumn names, 
	col sep=comma, 
	fixed zerofill,
	fixed,
	display columns/0/.style={int detect, 
		column name=Density,string type, 
	},  
	display columns/1/.style={precision=2,
		column name=$lb_{SDP}$,
	},
	display columns/2/.style={precision=2,
		column name=$lb_{DNN}$,
	},
	display columns/3/.style={string type,
		column name=Improvement,
	},
	display columns/4/.style={precision=0,
		column name=$ub$,
	},
	display columns/5/.style={precision=2,
		column name=Gap,
	},
	every head row/.style={
		before row={
			\caption{Feasible solutions for the randomly generated graphs rand80, rand50, rand20 ($n=100$ $k=5$, with an increased time limit for heuristics of 10 seconds)}\label{tab:keq_rand_100_5}\\
			\midrule}, 
		after row={
			\si{\percent}& &  & \si{\percent} & &\si{\percent}\\
			\endfirsthead
			\midrule} 
	},	
	every last row/.style={after row=\bottomrule}, 
	]{admm_keq_rand_compare_n100_k5.csv} 

%% file: GPKCrandcsv.tex
 
	\begin{filecontents*}{rand_GPKC_sdp_dnn.csv}
n,W,sdp,dnn,improvement_dnn,dnn+met,improvement_dnnmet
100,26536,79801.9019,79863.8948,0.08,80641.37746,1.05
,14023,117886.57529,121989.7238,3.48,122471.6871,3.89
,11434,125769.0795,132094.0787,5.03,132441.0324,5.3
,6321,141338.31974,154560.71,9.36,154733.5567,9.48
,3668,149417.80277,169471.951,13.42,169667.9708,13.55
,3043,151321.27971,173607.3459,14.73,173792.4458,14.85
200,50084,334496.24511,334909.6751,0.12,336293.7191,0.54
,25787,489900.71973,511836.1111,4.48,512694.324,4.65
,21741,515782.13353,544827.7458,5.63,545325.2253,5.73
,11404,581909.21145,640300.7628,10.03,640580.6399,10.08
,6686,612092.40076,692857.8327,13.19,693182.8716,13.25
,3419,632993.42357,737482.6145,16.51,737654.491,16.53
\end{filecontents*}
\pgfplotstabletypeset[  
		font=\footnotesize,
	multicolumn names, 
	col sep=comma, 
	precision =2,
	fixed,
	fixed zerofill,
	display columns/0/.style={int detect,
		column name=$n$, 
		},  
	display columns/1/.style={string type,
		column name=$W$,
		},
	display columns/2/.style={
		column name=$lb_{SDP}$,
		},
	display columns/3/.style={
		column name=$lb_{DNN}$,
		},
	display columns/4/.style={
		column name=$\textrm{Imp}~\%$,
		},
	display columns/5/.style={
		column name=$lb_{DNN+MET}$,
	},
	display columns/6/.style={
		column name=$\textrm{Imp}~\%$,
	},
	every row no 6/.style={before row={\midrule}},
	every head row/.style={fixed,
		before row={
			\caption{GPKC lower bounds on GPKCrand80}\label{tab:gpkc-lb}\\
			\midrule}, 
		after row={
			\midrule \endfirsthead
		} 
	},
	every last row/.style={after row=\bottomrule}, 
	]{rand_GPKC_sdp_dnn.csv} 

%% file: gpkctimecsv.tex
	\begin{filecontents*}{admm_GPKC_rand_time.csv}
n,W,Dnn iterations,Dnn CPUTIME,mosek cputime
100,26536,,,53.52
,14023,2167,36.59,51.07
,11434,2241,39.73,56.48
,6321,2369,42.02,47
,3668,2760,47.8,47.82
,3043,3198,55.3,45.55
200,50084,4918,315.4,3576.52
,25787,3257,207.93,3541.04
,21741,3431,217.94,3065.87
,11404,3760,241.98,2657.11
,6686,3564,225.52,2620.47
,3419,3538,220.82,2155.3
300,73485,,,
,49920,4837,557.46,
,31134,5798,678.43,
,25752,5820,676.88,
,16767,5191,602.28,
,6647,4010,459.82,
,4210,4728,541.19,
,2454,5355,606.81,
400,73485,,,
,49920,4837,557.46,
,31134,5798,678.43,
,25752,5820,676.88,
,16767,5191,602.28,
,6647,4010,459.82,
,4210,4728,541.19,
,2454,5355,606.81,
500,132135,,,
,53186,10061,3492.13,
,26965,8710,3237.27,
,15049,8982,3146.47,
,12806,9834,4426.28,
,7312,7054,3101.58,
,4071,6858,2442.62,
\end{filecontents*}
\pgfplotstabletypeset[  
	font=\footnotesize,
	multicolumn names, 
	col sep=comma, 
	fixed zerofill,
	fixed,
	display columns/0/.style={int detect,
		column name=$n$, 
		},  
	display columns/1/.style={int detect, 
		column name=$W$,
		},
	display columns/2/.style={empty cells with={--},string type,
		column name=$\textrm{Iterations}$,
		},
	display columns/3/.style={empty cells with={--},
		column name=$\textrm{CPU time}$,
	},
	display columns/4/.style={
		column name=$\textrm{CPU time}$,empty cells with={--}, 
	},
	every row no 6/.style={before row={\midrule}},
	every row no 12/.style={before row={\midrule}},
	every row no 20/.style={before row={\midrule}},
	every row no 28/.style={before row={\midrule}},
	every row no 35/.style={before row={\midrule}},
	every row no 45/.style={before row={\midrule}},
	every row no 53/.style={before row={\midrule}},
	every row no 61/.style={before row={\midrule}},
	every row no 69/.style={before row={\toprule}},
	every head row/.style={precision=2,
		before row={
		\caption{Computation times for GPKC problems}\label{tab:gpkc_time}\\
		\toprule 
		& & \multicolumn{2}{c}{extended ADMM} & Mosek \\
		& & \multicolumn{2}{c}{DNN \eqref{eq:gpkc-sdp}} & DNN \eqref{eq:gpkc-sdp} \\},
		after row={
		& & &(\si{\second}) &(\si{\second})  \\ 
		\midrule 
		\endfirsthead
	} 
		},
	every last row/.style={after row=\bottomrule}, 
	]{admm_GPKC_rand_time.csv} 

%% file: GPKCrandupcsv.tex
	\begin{filecontents*}{rand_GPKC_dnn_ub.csv}
n,W,DNN,nc+2opt(5s),gap
100,26536,79861.22773,81191,1.66510371527944
,14023,121989.2193,124781,2.28854706671609
,11434,132093.7089,135681,2.71571684213646
,6321,154560.5607,157773,2.07843403611566
,3668,169471.8938,172343,1.69414888547142
,3043,173607.318,176385,1.59997978887042
200,50084,334909.3443,339970,1.51105240451782
,25787,511835.526,517093,1.02718036027847
,21741,544827.4478,553988,1.68136760308059
,11404,640300.8162,651955,1.82011072063975
,6686,692857.7922,704127,1.6264820756677
,3419,737482.4524,745057,1.02708173941632
300,73485,723712.9656,729625,0.82
,49920,992357.1168,1004965,1.27
,31134,1229623.388,1249143,1.59
,25752,1304759.476,1325865,1.62
,16767,1440209.558,1462605,1.56
,6647,1618112.756,1638290,1.25
,4210,1672195.756,1687602,0.92
,2454,1719157.251,1729916,0.63
400,99348,1338105.41,1351240,0.98
,51033,2071963.021,2093269,1.03
,41130,2237953.682,2264663,1.19
,28718,2462598.438,2494382,1.29
,22740,2578443.431,2614666,1.4
,11772,2812875.09,2848815,1.28
,6132,2961922.931,2991521,1
500,132135,1936655.663,1986415,2.57
,53186,3467473.904,3521678,1.56
,26965,4073265.275,4130236,1.4
,15049,4396228.102,4452115,1.27
,12806,4464071.029,4523570,1.33
,7312,4651268.182,4697755,1
,4071,4783543.507,4821216,0.79
\end{filecontents*}
\pgfplotstabletypeset[  
	font=\footnotesize,
	multicolumn names, 
	col sep=comma, 
	fixed zerofill,
	fixed,
	display columns/0/.style={int detect,
		column name=$n$, 
		},  
	display columns/1/.style={int detect,
		column name=$W$,
		},
	display columns/2/.style={precision=2,
		column name=$lb_{DNN}$,
		},
	display columns/3/.style={int detect,
		column name=VC+2opt,
		},
	display columns/4/.style={
		column name=Gap,
		},
	every row no 6/.style={before row={\midrule}},
	every row no 12/.style={before row={\midrule}},
	every row no 20/.style={before row={\midrule}},
	every row no 27/.style={before row={\midrule}},
	every head row/.style={
		before row={
			\caption{Feasible solutions for randomly generated graphs on GPKC problems GPKCrand80 (the maximum number of iterations for eADMM is 20\;000, a $^*$ indicates that the maximum number of iterations is reached)}\label{tab:gpkc_ub}\\
			\midrule}, 
		after row={
		 &  &  & &\si{\percent}  \\
		 		 \endfirsthead
		  	\midrule} 
	},	
	row iteratorMaxiter/.style={
		every row #1 column 2/.style={
			postproc cell content/.append style={
				/pgfplots/table/@cell content/.add={$^*$}}
		}
	},	
	row iteratorMaxiter/.list={0,12,27,},
	every last row/.style={after row=\bottomrule}, 
	]{rand_GPKC_dnn_ub.csv}
	

%% file: GPKCrandupcsv_50.tex
	\begin{filecontents*}{rand_GPKC_dnn_ub_50.csv}
n,W,DNN,nc+2opt(5s),gap
100,27572,43862.44799,46578,6.19106350520861
,14639,69143.3557,73148,5.79179916776877
,11686,75761.36692,80014,5.61319476256356
,6526,89183.5034,93252,4.56193852550538
,4040,97547.91447,101998,4.56194840676843
,3179,101184.1191,105129,3.89871546551815
200,48935,193293.8768,196369,1.59090564631895
,25685,300462.9007,310368,3.29661308498444
,21504,322166.703,333111,3.39709128786038
,10742,384833.045,397883,3.39106923627102
,6034,418763.8352,432026,3.1669794966096
,3940,438101.1489,449402,2.57950729606042
300,76829,433450.1167,449273,3.65
,49664,615471.2949,626291,1.76
,32370,741306.3299,763853,3.04
,27123,782740.1652,804648,2.8
,16670,873377.6322,898159,2.84
,6371,985628.1235,1011695,2.64
,4428,1014899.089,1038083,2.28
,2649,1049198.04,1066520,1.65
400,97431,780703.5614,790271,1.23
,48514,1242772.574,1261840,1.53
,44456,1285533.556,1317925,2.52
,25558,1501960.489,1535376,2.22
,22264,1544147.973,1578729,2.24
,11477,1700761.302,1740466,2.33
,6800,1786701.777,1825146,2.15
500,131286,1289330.602,1301301,0.93
,56101,2127940.947,2185271,2.69
,29244,2476801.013,2534746,2.34
,15747,2685115.43,2750968,2.45
,13176,2730501.041,2793996,2.33
,7920,2836940.859,2893907,2.01
,4249,2937039.588,2979181,1.43
\end{filecontents*}
\pgfplotstabletypeset[  
	font=\footnotesize,
	multicolumn names, 
	col sep=comma, 
	fixed zerofill,
	fixed,
	display columns/0/.style={int detect,
		column name=$n$, 
		},  
	display columns/1/.style={int detect,
		column name=$W$,
		},
	display columns/2/.style={
		column name=$lb_{DNN}$,
		},
	display columns/3/.style={int detect,
		column name=VC+2opt,
		},
	display columns/4/.style={
		column name=Gap,
		},
	every row no 6/.style={before row={\midrule}},
	every row no 12/.style={before row={\midrule}},
	every row no 20/.style={before row={\midrule}},
	every row no 27/.style={before row={\midrule}},
	every head row/.style={
		before row={
			\caption{Feasible solutions for randomly generated graphs on GPKC problems GPKCrand50 (the maximum number of iterations for eADMM is 20\;000, a $^*$ indicates that the maximum number of iterations is reached)}\label{tab:gpkc_ub_50}\\
			\midrule}, 
		after row={
		 &  &  & &\si{\percent}  \\
		 		 \endfirsthead
		  	\midrule} 
	},	
	row iteratorMaxiter/.style={
		every row #1 column 2/.style={
			postproc cell content/.append style={
				/pgfplots/table/@cell content/.add={$^*$}}
		}
	},	
	row iteratorMaxiter/.list={0,1,6,7,8,10,11,12,20,27},
	every last row/.style={after row=\bottomrule}, 
	]{rand_GPKC_dnn_ub_50.csv}

%% file: GPKCrandupcsv_20.tex
	\begin{filecontents*}{rand_GPKC_dnn_ub_20.csv}
n,W,DNN(maxiter50000),nc+2opt(5s),gap
100,26472,14332.53668,16215,13.13
,14310,22340.23194,25531,14.28
,11397,24666.71678,27764,12.56
,6349,29856.16042,33447,12.03
,3973,33419.32037,36909,10.44
,3082,35224.71276,38693,9.85
200,49329,66505.12239,72779,9.43
,26138,104505.0405,114208,9.28
,21312,113363.638,124094,9.47
,11665,133889.0995,146638,9.52
,7007,147185.2816,160241,8.87
,3747,160643.0694,171998,7.07
300,73762,163996.9721,172803,5.37
,49938,222399.9589,236916,6.53
,31051,272370.3969,293840,7.88
,25979,287194.0712,309615,7.81
,16089,319402.5711,344106,7.73
,6616,361919.9944,386588,6.82
,4277,378355.7234,400646,5.89
,2829,392682.3438,411978,4.91
400,96066,295262.5042,306214,3.71
,49622,458045.1149,481621,5.15
,40512,493274.5233,520661,5.55
,26495,553117.7515,583029,5.41
,22338,572857.5517,606241,5.83
,11274,635081.6538,672230,5.85
,7300,664905.8602,702511,5.66
500,133561,470275.8942,490915,4.39
,54084,801360.8768,847836,5.8
,28866,925818.1373,980463,5.9
,14840,1012359.755,1073619,6.05
,12923,1026775.879,1086881,5.85
,7468,1076226.503,1132106,5.19
,3975,1123555.873,1170459,4.17
\end{filecontents*}
\pgfplotstabletypeset[  
	font=\footnotesize,
	multicolumn names, 
	col sep=comma, 
	font=\footnotesize,
	fixed,
	fixed zerofill,
	display columns/0/.style={int detect,
		column name=$n$, 
		},  
	display columns/1/.style={int detect,
		column name=$W$,
		},
	display columns/2/.style={
		column name=$lb_{DNN}$,
		},
	display columns/3/.style={int detect,
		column name=VC+2opt,
		},
	display columns/4/.style={
		column name=Gap,
		},
	every row no 6/.style={before row={
		\midrule}},
		every row no 12/.style={before row={\midrule}},
	every row no 20/.style={before row={\midrule}},
	every row no 27/.style={before row={\midrule}},
	row iteratorMaxiter/.style={
		every row #1 column 2/.style={
			postproc cell content/.append style={
				/pgfplots/table/@cell content/.add={$^*$}}
		}
	},	
	row iteratorMaxiter/.list={0,1,2,3,5,6,7,8,9,11,12,13,14,15,16,20,21,22,23,24,25,27},
	every last row/.style={after row=\bottomrule}, 
	every head row/.style={
		before row={
		\caption{Feasible solutions for randomly generated graphs on GPKC problems GPKCrand20 (the maximum number of iterations for eADMM is 50\;000, a $^*$ indicates that the maximum number of iterations is reached)}\label{tab:gpkc_ub_20}\\	
		\midrule}, 
		after row={
		 &  &   & &\si{\percent}\\
		 		 \endfirsthead
		  	\midrule} 
	},	
	every last row/.style={after row=\bottomrule}, 
	]{rand_GPKC_dnn_ub_20.csv}